\documentclass[a4paper, 12pt]{amsart}
\usepackage{amssymb, amsmath, amsthm, color, mathrsfs, mathtools}
\usepackage[colorlinks=true, breaklinks=true, linkcolor=black, citecolor=blue, urlcolor=red]{hyperref} 
\usepackage[english]{babel}
\usepackage{tikz-cd}%For commutative diagrams
\usepackage{tikz}
\usetikzlibrary{shapes}

\usepackage{enumitem}%To get label referencing working properly for enumerate

\usepackage{caption}
\usepackage{floatrow}   
\usepackage{subcaption}

\numberwithin{equation}{section}
\newtheorem{theorem}{Theorem}[section]

\newtheorem{lemma}[theorem]{Lemma}

\newtheorem{proposition}[theorem]{Proposition}

\newtheorem{observation}[theorem]{Observation}
\newtheorem{notation}[theorem]{Notation}
\newtheorem{example}[theorem]{Example}

\newtheorem*{namedthm}{\namedthmname}
\newcounter{namedthm}

\makeatletter
\newenvironment{named}[1]
  {\def\namedthmname{#1}%
   \refstepcounter{namedthm}%
   \namedthm\def\@currentlabel{#1}}
  {\endnamedthm}
\makeatother
\theoremstyle{definition} 
\newtheorem{definition}[theorem]{Definition}
\newtheorem{remark}[theorem]{Remark}

\newcommand{\act}{\curvearrowright}

\DeclareMathOperator{\Aff}{Aff}

\newcommand{\bP}{\mathbf{P}}
\newcommand{\al}{\alpha}
\newcommand{\cA}{\mathcal{A}}

\newcommand{\cB}{\mathcal{B}}
\newcommand{\cC}{\mathcal C}
\newcommand{\fC}{\mathfrak C}

\newcommand{\cD}{\mathcal D}

\newcommand{\cF}{\mathcal F}

\DeclareMathOperator{\Frac}{Frac}

\DeclareMathOperator{\FS}{FS}

\newcommand{\cG}{\mathcal G}
\DeclareMathOperator{\Gr}{Gr}

\newcommand{\ga}{\gamma}
\newcommand{\Ga}{\Gamma}

\DeclareMathOperator{\Homeo}{Homeo}

\DeclareMathOperator{\id}{id}
\newcommand{\into}{\hookrightarrow}

\newcommand{\la}{\langle}

\DeclareMathOperator{\Leaf}{Leaf}

\newcommand{\N}{\mathbf{N}}

\newcommand{\onto}{\twoheadrightarrow}
\newcommand{\ot}{\otimes}
\newcommand{\ov}{\overline}

\DeclareMathOperator{\PSL}{\textnormal{PSL}}
\newcommand{\PL}{\textnormal{PL}}
\newcommand{\PP}{\textnormal{PP}}

\newcommand{\Q}{\mathbf{Q}}

\newcommand{\fQ}{\mathfrak{Q}}
\newcommand{\R}{\mathbf{R}}

\newcommand{\ra}{\rangle}

\newcommand{\bS}{\mathbf{S}}

\DeclareMathOperator{\SL}{\textnormal{SL}}

\DeclareMathOperator{\supp}{supp}
\newcommand{\fT}{\mathfrak T}

\newcommand{\ti}{\tilde}

\newcommand{\Z}{\mathbf{Z}}

\newcommand{\0}{{\tt0}}
\newcommand{\1}{{\tt1}}

\newcommand{\PermutationA}{
	\begin{tikzpicture}
		\draw[thick] (1,1) -- (0,2);
		\draw[thick] (0,1) -- (2,2);
		\draw[thick] (2,1) -- (1,2);
		\draw[thick] (0,0) -- (1,1);
		\draw[thick] (1,0) -- (0,1);
		\draw[thick] (2,0) -- (2,1);
		\draw[fill=black] (0,0) circle [radius=0.06];
		\draw[fill=black] (1,0) circle [radius=0.06];
		\draw[fill=black] (2,0) circle [radius=0.06];
		\draw[fill=black] (0,1) circle [radius=0.06];
		\draw[fill=black] (1,1) circle [radius=0.06];
		\draw[fill=black] (2,1) circle [radius=0.06];
		\draw[fill=black] (0,2) circle [radius=0.06];
		\draw[fill=black] (1,2) circle [radius=0.06];
		\draw[fill=black] (2,2) circle [radius=0.06];
	\end{tikzpicture}
}
\newcommand{\PermutationB}{
	\begin{tikzpicture}
		\draw[thick] (0,0) -- (0,1);
		\draw[thick] (1,0) -- (2,1);
		\draw[thick] (2,0) -- (1,1);
		\draw[fill=black] (0,0) circle [radius=0.06];
		\draw[fill=black] (1,0) circle [radius=0.06];
		\draw[fill=black] (2,0) circle [radius=0.06];
		\draw[fill=black] (0,1) circle [radius=0.06];
		\draw[fill=black] (1,1) circle [radius=0.06];
		\draw[fill=black] (2,1) circle [radius=0.06];
	\end{tikzpicture}
}

\newcommand{\LeftVine}{
\begin{tikzpicture}
\draw[thick] (3,0) -- (2.5,0.5);
\draw[thick] (3,0) -- (3.5,0.5);
\draw[thick] (2.5,0.5) -- (2,1);
\draw[thick] (2.5,0.5) -- (3,1);
\draw[thick] (2,1) -- (1.5,1.5);
\draw[thick] (2,1) -- (2.5,1.5);
\draw[fill=black] (2.5,0.5) circle [radius=0.05];
\draw[fill=black] (3,1) circle [radius=0.05];
\draw[fill=black] (1.5,1.5) circle [radius=0.05];
\draw[fill=black] (2.5,1.5) circle [radius=0.05];
\draw[fill=black] (3,0) circle [radius=0.05];
\draw[fill=black] (3.5,0.5) circle [radius=0.05];
\draw[fill=black] (2,1) circle [radius=0.05];
\end{tikzpicture}
}

\newcommand{\RightVine}{
\begin{tikzpicture}
\draw[thick] (3,0) -- (2.5,0.5);
\draw[thick] (3,0) -- (3.5,0.5);
\draw[thick] (3.5,0.5) -- (3,1);
\draw[thick] (3.5,0.5) -- (4,1);
\draw[thick] (4,1) -- (3.5,1.5);
\draw[thick] (4,1) -- (4.5,1.5);
\draw[fill=black] (2.5,0.5) circle [radius=0.05];
\draw[fill=black] (3,1) circle [radius=0.05];
\draw[fill=black] (3.5,1.5) circle [radius=0.05];
\draw[fill=black] (4.5,1.5) circle [radius=0.05];
\draw[fill=black] (3,0) circle [radius=0.05];
\draw[fill=black] (3.5,0.5) circle [radius=0.05];
\draw[fill=black] (4,1) circle [radius=0.05];
\end{tikzpicture}
}

\newcommand{\Cleary}{
\begin{tikzpicture}
\draw[thick] (0+3,0) -- (-0.5+3,0.75);
\draw[thick] (0+3,0) -- (0.5+3,0.75);
\draw[thick] (0.5+3,0.75) -- (1+3,1.5);
\draw[thick] (0.5+3,0.75) -- (0+3,1.5);
\draw[black,fill=white,thick] (0+3,0) circle [radius=0.25] node {$a$};
\draw[black,fill=white,thick] (0.5+3,0.75) circle [radius=0.25] node {$a$};
\draw[fill=black] (1+3,1.5) circle [radius=0.06];
\draw[fill=black] (0+3,1.5) circle [radius=0.06];
\draw[fill=black] (-0.5+3,0.75) circle [radius=0.06];
\node at (1.5,0) {$\sim$};
\draw[thick] (0,0) -- (-0.5,0.75);
\draw[thick] (0,0) -- (0.5,0.75);
\draw[thick] (-0.5,0.75) -- (-1,1.5);
\draw[thick] (-0.5,0.75) -- (0,1.5);
\draw[black,fill=white,thick] (0,0) circle [radius=0.25] node {$b$};
\draw[black,fill=white,thick] (-0.5,0.75) circle [radius=0.25] node {$b$};
\draw[fill=black] (-1,1.5) circle [radius=0.06];
\draw[fill=black] (0,1.5) circle [radius=0.06];
\draw[fill=black] (0.5,0.75) circle [radius=0.06];
\end{tikzpicture}  
}

\newcommand{\ForestA}{
\begin{tikzpicture}
\draw[fill=black] (0,0) circle [radius=0.05];
\draw[thick] (0,0) -- (-0.5,0.5);
\draw[thick] (0,0) -- (0.5,0.5);
\draw[fill=black] (-0.5,0.5) circle [radius=0.05];
\draw[thick] (1,-0.2) -- (1,0.5);
\draw[fill=black] (1,-0.2) circle [radius=0.05];
\draw[fill=black] (1,0.5) circle [radius=0.05];
\draw[fill=black] (0.5,0.5) circle [radius=0.05];
\draw[black,fill=white,thick] (0,0) circle [radius=0.25] node {$a$};
\end{tikzpicture}
}

\newcommand{\ActionOne}{
\begin{tikzpicture}[scale=0.8]
		\draw[thick] (0,0) -- (0.7,0.75);
		\draw[thick] (0,0) -- (-0.7,0.75);
		\draw[black,fill=white,thick] (0,0) circle [radius=0.25] node {$b$};
		\draw[fill=black] (-0.7,0.75) circle [radius=0.06];
		\draw[fill=black] (0.7,0.75) circle [radius=0.06];
		\draw[thick] (-0.7,0.75) -- (0.7,1.75);
		\draw[thick] (0.7,0.75) -- (-0.7,1.75);
		\draw[thick] (0.7,1.75) -- (0,2.5);
		\draw[thick] (-0.7,1.75) -- (0,2.5);
        \draw[thick] (0,2.5)--(0,3.25);
		\draw[black,fill=white,thick] (0,2.5) circle [radius=0.25] node {$a$};
		\draw[fill=black] (0.7,1.75) circle [radius=0.06];
		\draw[fill=black] (-0.7,1.75) circle [radius=0.06];
        \draw[thick] (0,3.25)--(-0.7,4);
        \draw[thick] (0,3.25)--(0.7,4);
        \draw[black,fill=white,thick] (0,3.25) circle [radius=0.25] node {$a$};
        \draw[fill=black] (0.7,4) circle [radius=0.06];
		\draw[fill=black] (-0.7,4) circle [radius=0.06];
        \node at (-0.7,4.25) {$x$};
	\end{tikzpicture}
}

\newcommand{\ActionTwo}{
\begin{tikzpicture}[scale=0.8]
		\draw[thick] (0,0) -- (0.7,0.75);
		\draw[thick] (0,0) -- (-0.7,0.75);
		\draw[black,fill=white,thick] (0,0) circle [radius=0.25] node {$b$};
		\draw[fill=black] (-0.7,0.75) circle [radius=0.06];
		\draw[fill=black] (0.7,0.75) circle [radius=0.06];
		\draw[thick] (-0.7,0.75) -- (0.7,1.75);
		\draw[thick] (0.7,0.75) -- (-0.7,1.75);
        \draw[fill=black] (-0.7,1.75) circle [radius=0.06];
		\draw[fill=black] (0.7,1.75) circle [radius=0.06];
        \node at (-0.7,2) {$x$};
	\end{tikzpicture}
}

\newcommand{\ActionThree}{
\begin{tikzpicture}[scale=0.8]
		\draw[thick] (0,0) -- (0.7,0.75);
		\draw[thick] (0,0) -- (-0.7,0.75);
		\draw[black,fill=white,thick] (0,0) circle [radius=0.25] node {$b$};
		\draw[fill=black] (-0.7,0.75) circle [radius=0.06];
		\draw[fill=black] (0.7,0.75) circle [radius=0.06];
        \node at (0.7,1) {$x$};
	\end{tikzpicture}
}

\newcommand{\FigD}{
\begin{tikzpicture}
\draw[fill=black] (0,0) circle [radius=0.05];
\draw[thick] (0,0) -- (-0.5,0.5);
\draw[thick] (0,0) -- (0.5,0.5);
\draw[fill=black] (-0.5,0.5) circle [radius=0.05];
\draw[thick] (-0.5,0.5) -- (-1,1);
\draw[thick] (-0.5,0.5) -- (0,1);
\draw[fill=black] (-1,1) circle [radius=0.05];
\draw[fill=black] (0,1) circle [radius=0.05];
\draw[fill=black] (0.5,0.5) circle [radius=0.05];
\draw[fill=black] (2,0) circle [radius=0.05];
\draw[thick] (2,0) -- (1.5,0.5);
\draw[thick] (2,0) -- (2.5,0.5);
\draw[fill=black] (1.5,0.5) circle [radius=0.05];
\draw[fill=black] (2.5,0.5) circle [radius=0.05];
\draw[fill=black] (4,-0.125) circle [radius=0.05];
\draw[fill=black] (4,0.5) circle [radius=0.05];
\draw[thick] (4,-0.125) -- (4,0.5);
\draw[black,fill=white,thick] (0,0) circle [radius=0.25] node {$a$};
\draw[black,fill=white,thick] (-0.5,0.5) circle [radius=0.25] node {$b$};
\draw[black,fill=white,thick] (2,0) circle [radius=0.25] node {$b$};
\end{tikzpicture}
}

\newcommand{\Composition}{
\begin{tikzpicture}
\draw[fill=black] (0,0) circle [radius=0.05];
\draw[thick] (0,0) -- (-0.5,0.5);
\draw[thick] (0,0) -- (0.5,0.5);
\draw[fill=black] (-0.5,0.5) circle [radius=0.05];
\draw[thick] (1,-0.2) -- (1,0.5);
\draw[thick] (1.5,-0.2) -- (1.5,0.5);
\draw[fill=black] (1,-0.2) circle [radius=0.05];
\draw[fill=black] (1,0.5) circle [radius=0.05];
\draw[fill=black] (1.5,-0.2) circle [radius=0.05];
\draw[fill=black] (1.5,0.5) circle [radius=0.05];
\draw[fill=black] (0.5,0.5) circle [radius=0.05];
\draw[black,fill=white,thick] (0,0) circle [radius=0.25] node {$a$};
\node at (2,0.15) {$\circ$};
\draw[thick] (2.5,-0.2) -- (2.5,0.5);
\draw[thick] (3,-0.2) -- (3,0.5);
\draw[fill=black] (2.5,-0.2) circle [radius=0.05];
\draw[fill=black] (2.5,0.5) circle [radius=0.05];
\draw[fill=black] (3,-0.2) circle [radius=0.05];
\draw[fill=black] (3,0.5) circle [radius=0.05];
\draw[fill=black] (4,0) circle [radius=0.05];
\draw[thick] (4,0) -- (3.5,0.5);
\draw[thick] (4,0) -- (4.5,0.5);
\draw[fill=black] (3.5,0.5) circle [radius=0.05];
\draw[fill=black] (4.5,0.5) circle [radius=0.05];
\draw[black,fill=white,thick] (4,0) circle [radius=0.25] node {$b$};
\draw[thick] (5,-0.2) -- (5,0.5);
\draw[fill=black] (5,-0.2) circle [radius=0.05];
\draw[fill=black] (5,0.5) circle [radius=0.05];
\node at (5.625,0.1) {$=$};
\draw[fill=black] (6.5,0) circle [radius=0.05];
\draw[thick] (6.5,0) -- (6.25,0.5);
\draw[thick] (6.5,0) -- (6.75,0.5);
\draw[thick] (7.5,-0.2) -- (7.5,0.5);
\draw[fill=black] (7.5,-0.2) circle [radius=0.05];
\draw[fill=black] (6.25,0.5) circle [radius=0.05];
\draw[fill=black] (6.75,0.5) circle [radius=0.05];
\draw[black,fill=white,thick] (6.5,0) circle [radius=0.25] node {$a$};
\draw[thick] (7.5,0.5) -- (7.25,1);
\draw[thick] (7.5,0.5) -- (7.75,1);
\draw[thick] (6.25,0.5) -- (6.25,1);
\draw[thick] (6.75,0.5) -- (6.75,1);
\draw[fill=black] (6.25,1) circle [radius=0.05];
\draw[fill=black] (6.75,1) circle [radius=0.05];
\draw[fill=black] (7.25,1) circle [radius=0.05];
\draw[fill=black] (7.75,1) circle [radius=0.05];
\draw[black,fill=white,thick] (7.5,0.5) circle [radius=0.25] node {$b$};
\draw[thick] (8.5,-0.2) -- (8.5,0.5);
\draw[thick] (8.5,0.5) -- (8.5,1);
\draw[fill=black] (8.5,-0.2) circle [radius=0.05];
\draw[fill=black] (8.5,0.5) circle [radius=0.05];
\draw[fill=black] (8.5,1) circle [radius=0.05];
\node at (9.125,0.1) {$=$};
\draw[fill=black] (10,0) circle [radius=0.05];
\draw[thick] (10,0) -- (9.75,0.5);
\draw[thick] (10,0) -- (10.25,0.5);
\draw[fill=black] (9.75,0.5) circle [radius=0.05];
\draw[fill=black] (10.25,0.5) circle [radius=0.05];
\draw[black,fill=white,thick] (10,0) circle [radius=0.25] node {$a$};
\draw[thick] (11,0) -- (10.75,0.5);
\draw[thick] (11,0) -- (11.25,0.5);
\draw[fill=black] (10.75,0.5) circle [radius=0.05];
\draw[fill=black] (11.25,0.5) circle [radius=0.05];
\draw[black,fill=white,thick] (11,0) circle [radius=0.25] node {$b$};
\draw[thick] (11.75,-0.2) -- (11.75,0.5);
\draw[fill=black] (11.75,0.5) circle [radius=0.05];
\draw[fill=black] (11.75,-0.2) circle [radius=0.05];
\end{tikzpicture}
}

\newcommand{\TensorProd}{
\begin{tikzpicture}
\draw[fill=black] (0,0) circle [radius=0.05];
\draw[thick] (0,0) -- (-0.5,0.5);
\draw[thick] (0,0) -- (0.5,0.5);
\draw[fill=black] (-0.5,0.5) circle [radius=0.05];
\draw[thick] (1,-0.2) -- (1,0.5);
\draw[fill=black] (1,-0.2) circle [radius=0.05];
\draw[fill=black] (1,0.5) circle [radius=0.05];
\draw[fill=black] (0.5,0.5) circle [radius=0.05];
\draw[black,fill=white,thick] (0,0) circle [radius=0.25] node {$a$};
\node at (1.75,0.15) {$\ot$};
\draw[thick] (2.5,-0.2) -- (2.5,0.5);
\draw[thick] (3,-0.2) -- (3,0.5);
\draw[fill=black] (2.5,-0.2) circle [radius=0.05];
\draw[fill=black] (2.5,0.5) circle [radius=0.05];
\draw[fill=black] (3,-0.2) circle [radius=0.05];
\draw[fill=black] (3,0.5) circle [radius=0.05];
\draw[fill=black] (4,0) circle [radius=0.05];
\draw[thick] (4,0) -- (3.5,0.5);
\draw[thick] (4,0) -- (4.5,0.5);
\draw[fill=black] (3.5,0.5) circle [radius=0.05];
\draw[fill=black] (4.5,0.5) circle [radius=0.05];
\draw[black,fill=white,thick] (4,0) circle [radius=0.25] node {$b$};
\node at (5,0.1) {$=$};
\draw[thick] (6,0) -- (5.5,0.5);
\draw[thick] (6,0) -- (6.5,0.5);
\draw[fill=black] (5.5,0.5) circle [radius=0.05];
\draw[fill=black] (6.5,0.5) circle [radius=0.05];
\draw[thick] (7,-0.2) -- (7,0.5);
\draw[thick] (7.5,-0.2) -- (7.5,0.5);
\draw[thick] (8,-0.2) -- (8,0.5);
\draw[fill=black] (7,-0.2) circle [radius=0.05];
\draw[fill=black] (7,0.5) circle [radius=0.05];
\draw[fill=black] (7.5,-0.2) circle [radius=0.05];
\draw[fill=black] (7.5,0.5) circle [radius=0.05];
\draw[fill=black] (8,-0.2) circle [radius=0.05];
\draw[fill=black] (8,0.5) circle [radius=0.05];
\draw[black,fill=white,thick] (6,0) circle [radius=0.25] node {$a$};
\draw[thick] (9,0) -- (8.5,0.5);
\draw[thick] (9,0) -- (9.5,0.5);
\draw[fill=black] (8.5,0.5) circle [radius=0.05];
\draw[fill=black] (9.5,0.5) circle [radius=0.05];
\draw[black,fill=white,thick] (9,0) circle [radius=0.25] node {$b$};
\end{tikzpicture}
}

\newcommand{\SkeinRelationThree}{
\begin{tikzpicture}
\draw[fill=black] (0,0) circle [radius=0.05];
\draw[thick] (0,0) -- (-0.5,0.5);
\draw[thick] (0,0) -- (0.5,0.5);
\draw[thick] (-0.5,0.5) -- (-0.75,1);
\draw[thick] (-0.5,0.5) -- (-0.25,1);
\draw[thick] (0.5,0.5) -- (0.75,1);
\draw[thick] (0.5,0.5) -- (0.25,1);
\draw[fill=black] (-0.75,1) circle [radius=0.05];
\draw[fill=black] (-0.25,1) circle [radius=0.05];
\draw[fill=black] (0.75,1) circle [radius=0.05];
\draw[fill=black] (0.25,1) circle [radius=0.05];
\draw[black,fill=white,thick] (0,0) circle [radius=0.25] node {$a$};
\draw[black,fill=white,thick] (-0.5,0.5) circle [radius=0.25] node {$a$};
\draw[black,fill=white,thick] (0.5,0.5) circle [radius=0.25] node {$a$};
\node at (1.5,0) {$\sim$};
\draw[thick] (3,0) -- (2.5,0.5);
\draw[thick] (3,0) -- (3.5,0.5);
\draw[thick] (3.5,0.5) -- (3,1);
\draw[thick] (3.5,0.5) -- (4,1);
\draw[thick] (4,1) -- (3.5,1.5);
\draw[thick] (4,1) -- (4.5,1.5);
\draw[fill=black] (2.5,0.5) circle [radius=0.05];
\draw[fill=black] (3,1) circle [radius=0.05];
\draw[fill=black] (3.5,1.5) circle [radius=0.05];
\draw[fill=black] (4.5,1.5) circle [radius=0.05];
\draw[black,fill=white,thick] (3,0) circle [radius=0.25] node {$b$};
\draw[black,fill=white,thick] (3.5,0.5) circle [radius=0.25] node {$b$};
\draw[black,fill=white,thick] (4,1) circle [radius=0.25] node {$b$};
\end{tikzpicture}
}

\newcommand{\SkeinRelationFour}{
\begin{tikzpicture}
\draw[fill=black] (0,0) circle [radius=0.05];
\draw[thick] (0,0) -- (-0.5,0.5);
\draw[thick] (0,0) -- (0.5,0.5);
\draw[thick] (-0.5,0.5) -- (-0.875,1.125);
\draw[thick] (-0.5,0.5) -- (-0.25,1);
\draw[thick] (0.5,0.5) -- (0.75,1);
\draw[thick] (0.5,0.5) -- (0.25,1);
\draw[thick] (-0.875,1.125) -- (-1.125,1.625);
\draw[thick] (-0.875,1.125) -- (-0.625,1.625);
\draw[fill=black] (-1.125,1.625) circle [radius=0.05];
\draw[fill=black] (-0.625,1.625) circle [radius=0.05];
\draw[fill=black] (-0.25,1) circle [radius=0.05];
\draw[fill=black] (0.75,1) circle [radius=0.05];
\draw[fill=black] (0.25,1) circle [radius=0.05];
\draw[black,fill=white,thick] (0,0) circle [radius=0.25] node {$a$};
\draw[black,fill=white,thick] (-0.5,0.5) circle [radius=0.25] node {$a$};
\draw[black,fill=white,thick] (0.5,0.5) circle [radius=0.25] node {$a$};
\draw[black,fill=white,thick] (-0.875,1.125) circle [radius=0.25] node {$a$};
\node at (1.5,0) {$\sim$};
\draw[thick] (3,0) -- (2.5,0.5);
\draw[thick] (3,0) -- (3.5,0.5);
\draw[thick] (3.5,0.5) -- (3,1);
\draw[thick] (3.5,0.5) -- (4,1);
\draw[thick] (4,1) -- (3.5,1.5);
\draw[thick] (4,1) -- (4.5,1.5);
\draw[thick] (4.5,1.5) -- (4,2);
\draw[thick] (4.5,1.5) -- (5,2);
\draw[fill=black] (2.5,0.5) circle [radius=0.05];
\draw[fill=black] (3,1) circle [radius=0.05];
\draw[fill=black] (3.5,1.5) circle [radius=0.05];
\draw[fill=black] (4,2) circle [radius=0.05];
\draw[fill=black] (5,2) circle [radius=0.05];
\draw[black,fill=white,thick] (3,0) circle [radius=0.25] node {$b$};
\draw[black,fill=white,thick] (3.5,0.5) circle [radius=0.25] node {$b$};
\draw[black,fill=white,thick] (4,1) circle [radius=0.25] node {$b$};
\draw[black,fill=white,thick] (4.5,1.5) circle [radius=0.25] node {$b$};
\end{tikzpicture}
}

\newcommand{\SkeinRelationFive}{
\begin{tikzpicture}
\draw[fill=black] (0,0) circle [radius=0.05];
\draw[thick] (0,0) -- (-0.5,0.5);
\draw[thick] (0,0) -- (0.5,0.5);
\draw[thick] (-0.5,0.5) -- (-0.875,1.125);
\draw[thick] (-0.5,0.5) -- (-0.25,1);
\draw[thick] (0.5,0.5) -- (0.75,1);
\draw[thick] (0.5,0.5) -- (0.25,1);
\draw[thick] (-0.875,1.125) -- (-1.125,1.75);
\draw[thick] (-0.875,1.125) -- (-0.625,1.625);
\draw[thick] (-1.125,1.75) -- (-1.375,2.25);
\draw[thick] (-1.125,1.75) -- (-0.875,2.25);
\draw[fill=black] (-1.375,2.25) circle [radius=0.05];
\draw[fill=black] (-0.875,2.25) circle [radius=0.05];
\draw[fill=black] (-0.625,1.625) circle [radius=0.05];
\draw[fill=black] (-0.25,1) circle [radius=0.05];
\draw[fill=black] (0.75,1) circle [radius=0.05];
\draw[fill=black] (0.25,1) circle [radius=0.05];
\draw[black,fill=white,thick] (0,0) circle [radius=0.25] node {$a$};
\draw[black,fill=white,thick] (-0.5,0.5) circle [radius=0.25] node {$a$};
\draw[black,fill=white,thick] (0.5,0.5) circle [radius=0.25] node {$a$};
\draw[black,fill=white,thick] (-0.875,1.125) circle [radius=0.25] node {$a$};
\draw[black,fill=white,thick] (-1.125,1.75) circle [radius=0.25] node {$a$};
\node at (1.5,0) {$\sim$};
\draw[thick] (3,0) -- (2.5,0.5);
\draw[thick] (3,0) -- (3.5,0.5);
\draw[thick] (3.5,0.5) -- (3,1);
\draw[thick] (3.5,0.5) -- (4,1);
\draw[thick] (4,1) -- (3.5,1.5);
\draw[thick] (4,1) -- (4.5,1.5);
\draw[thick] (4.5,1.5) -- (4,2);
\draw[thick] (4.5,1.5) -- (5,2);
\draw[thick] (5,2) -- (4.5,2.5);
\draw[thick] (5,2) -- (5.5,2.5);
\draw[fill=black] (2.5,0.5) circle [radius=0.05];
\draw[fill=black] (3,1) circle [radius=0.05];
\draw[fill=black] (3.5,1.5) circle [radius=0.05];
\draw[fill=black] (4,2) circle [radius=0.05];
\draw[fill=black] (4.5,2.5) circle [radius=0.05];
\draw[fill=black] (5.5,2.5) circle [radius=0.05];
\draw[black,fill=white,thick] (3,0) circle [radius=0.25] node {$b$};
\draw[black,fill=white,thick] (3.5,0.5) circle [radius=0.25] node {$b$};
\draw[black,fill=white,thick] (4,1) circle [radius=0.25] node {$b$};
\draw[black,fill=white,thick] (4.5,1.5) circle [radius=0.25] node {$b$};
\draw[black,fill=white,thick] (5,2) circle [radius=0.25] node {$b$};
\end{tikzpicture}
}

\newcommand{\Cabelling}{
\begin{tikzpicture}
\draw[thick] (1,-1) -- (0,0);
\draw[thick] (0,-1) -- (1,0);
\draw[thick] (1.75,-1) -- (1.75,0);
\draw[fill=black] (0,-1) circle [radius=0.05];
\draw[fill=black] (1,-1) circle [radius=0.05];
\draw[fill=black] (1.75,-1) circle [radius=0.05];
\draw[fill=black] (0,0) circle [radius=0.05];
\draw[thick] (0,0) -- (-0.5,0.5);
\draw[thick] (0,0) -- (0.5,0.5);
\draw[fill=black] (-0.5,0.5) circle [radius=0.05];
\draw[thick] (1,0) -- (1,0.5);
\draw[thick] (1.75,0) -- (1.75,0.5);
\draw[fill=black] (1,0) circle [radius=0.05];
\draw[fill=black] (1,0.5) circle [radius=0.05];
\draw[fill=black] (1.75,0) circle [radius=0.05];
\draw[fill=black] (1.75,0.5) circle [radius=0.05];
\draw[fill=black] (0.5,0.5) circle [radius=0.05];
\draw[black,fill=white,thick] (0,0) circle [radius=0.25] node {$a$};
\node at (2.5,-0.75) {$=$};
\draw[thick] (3.25,-1) -- (3.25,-0.5);
\draw[thick] (4.25,-1) -- (3.75,-0.5);
\draw[thick] (4.25,-1) -- (4.75,-0.5);
\draw[thick] (5.25,-1) -- (5.25,-0.5);
\draw[fill=black] (3.25,-1) circle [radius=0.05];
\draw[fill=black] (3.25,-0.5) circle [radius=0.05];
\draw[fill=black] (3.75,-0.5) circle [radius=0.05];
\draw[fill=black] (4.75,-0.5) circle [radius=0.05];
\draw[fill=black] (5.25,-1) circle [radius=0.05];
\draw[fill=black] (5.25,-0.5) circle [radius=0.05];
\draw[thick] (3.75,-0.5) -- (3.25,0.5);
\draw[thick] (4.75,-0.5) -- (4.25,0.5);
\draw[thick] (3.25,-0.5) -- (4.75,0.5);
\draw[thick] (5.25,-0.5) -- (5.25,0.5);
\draw[fill=black] (3.25,0.5) circle [radius=0.05];
\draw[fill=black] (4.25,0.5) circle [radius=0.05];
\draw[fill=black] (4.75,0.5) circle [radius=0.05];
\draw[fill=black] (5.25,0.5) circle [radius=0.05];
\draw[black,fill=white,thick] (4.25,-1) circle [radius=0.25] node {$a$};
\end{tikzpicture}
}

\newcommand{\Cell}{
\begin{tikzpicture}
\draw[thick] (0,0) -- (-1,1);
\draw[fill=black] (0,0) circle [radius=0.04];
\draw[fill=black] (-1,1) circle [radius=0.04];
\draw[thick] (0,0) -- (1,1);
\draw[fill=black] (1,1) circle [radius=0.04];
\draw[thick] (-0.5,0.5) -- (0,1);
\draw[fill=black] (-0.5,0.5) circle [radius=0.04];
\draw[fill=black] (0,1) circle [radius=0.04];
\draw[thick] (-0.75,0.75) -- (-0.5,1);
\draw[fill=black] (-0.75,0.75) circle [radius=0.04];
\draw[fill=black] (-0.5,1) circle [radius=0.04];
\draw[thick] (0.75,0.75) -- (0.5,1);
\draw[fill=black] (0.75,0.75) circle [radius=0.04];
\draw[fill=black] (0.5,1) circle [radius=0.04];
\end{tikzpicture}
}

\newcommand{\InfiniteQuasiRightVine}{
\begin{tikzpicture}
\draw[thick] (0,0) -- (-1,1);
\draw[fill=black] (0,0) circle [radius=0.04];
\draw[fill=black] (-1,1) circle [radius=0.04];
\draw[thick] (0,0) -- (1,1);
\draw[thick] (-0.5,0.5) -- (0,1);
\draw[fill=black] (-0.5,0.5) circle [radius=0.04];
\draw[fill=black] (0,1) circle [radius=0.04];
\draw[thick] (-0.75,0.75) -- (-0.5,1);
\draw[fill=black] (-0.75,0.75) circle [radius=0.04];
\draw[fill=black] (-0.5,1) circle [radius=0.04];
\draw[thick] (0.75,0.75) -- (0.5,1);
\draw[fill=black] (0.75,0.75) circle [radius=0.04];
\draw[fill=black] (0.5,1) circle [radius=0.04];
\draw[thick] (0+1,0+1) -- (-1+1,1+1);
\draw[fill=black] (0+1,0+1) circle [radius=0.04];
\draw[fill=black] (-1+1,1+1) circle [radius=0.04];
\draw[thick] (0+1,0+1) -- (1+1,1+1);
\draw[thick] (-0.5+1,0.5+1) -- (0+1,1+1);
\draw[fill=black] (-0.5+1,0.5+1) circle [radius=0.04];
\draw[fill=black] (0+1,1+1) circle [radius=0.04];
\draw[thick] (-0.75+1,0.75+1) -- (-0.5+1,1+1);
\draw[fill=black] (-0.75+1,0.75+1) circle [radius=0.04];
\draw[fill=black] (-0.5+1,1+1) circle [radius=0.04];
\draw[thick] (0.75+1,0.75+1) -- (0.5+1,1+1);
\draw[fill=black] (0.75+1,0.75+1) circle [radius=0.04];
\draw[fill=black] (0.5+1,1+1) circle [radius=0.04];
\draw[thick] (0+2,0+2) -- (-1+2,1+2);
\draw[fill=black] (0+2,0+2) circle [radius=0.04];
\draw[fill=black] (-1+2,1+2) circle [radius=0.04];
\draw[thick] (0+2,0+2) -- (1+2,1+2);
\draw[thick] (-0.5+2,0.5+2) -- (0+2,1+2);
\draw[fill=black] (-0.5+2,0.5+2) circle [radius=0.04];
\draw[fill=black] (0+2,1+2) circle [radius=0.04];
\draw[thick] (-0.75+2,0.75+2) -- (-0.5+2,1+2);
\draw[fill=black] (-0.75+2,0.75+2) circle [radius=0.04];
\draw[fill=black] (-0.5+2,1+2) circle [radius=0.04];
\draw[thick] (0.75+2,0.75+2) -- (0.5+2,1+2);
\draw[fill=black] (0.75+2,0.75+2) circle [radius=0.04];
\draw[fill=black] (0.5+2,1+2) circle [radius=0.04];
\node[rotate=45] at (3.25,3.25) {$\cdots$};
\end{tikzpicture}
}

\newcommand{\InfinitePL}{
\begin{tikzpicture}[scale=8.75]

\draw[->] (0,0) -- (0,1);
\draw[->] (0,0) -- (1,0);

\draw[black,fill=black] (0,0) circle [radius=0.01];
\draw node[below left] at (0,0) {$0$};

\draw[line width=0.25mm,opacity=0.75] (0,0) -- (1/2,1/4);
\draw[line width=0.25mm,opacity=0.75] (1/2,1/4) -- (3/4,3/4);

\draw[black,fill=black] (1/2,1/4) circle [radius=0.01];
\draw[black,fill=black] (5/8,1/2) circle [radius=0.01];
\draw[black,fill=black] (3/4,3/4) circle [radius=0.01];

\draw[black,fill=black] (1/2,0) circle [radius=0.01];
\draw node[below] at (1/2,0) {$\frac{1}{2}$};\draw[black,fill=black] (5/8,0) circle [radius=0.01];
\draw node[below] at (5/8,0) {$\frac{5}{8}$};
\draw[black,fill=black] (3/4,0) circle [radius=0.01];
\draw node[below] at (3/4,0) {$\frac{3}{4}$};

\draw[black,fill=black] (0,1/4) circle [radius=0.01];
\draw node[anchor=east] at (0,1/4) {\scriptsize$1/4$};
\draw[black,fill=black] (0,1/2) circle [radius=0.01];
\draw node[anchor=east] at (0,1/2) {\scriptsize$1/2$};
\draw[black,fill=black] (0,3/4) circle [radius=0.01];
\draw node[anchor=east] at (0,3/4) {\scriptsize$3/4$};

\draw[line width=0.125mm,opacity=0.625] (3/4,3/4) -- (7/8,13/16);
\draw[line width=0.125mm,opacity=0.625] (7/8,13/16) -- (15/16,15/16);

\draw[black,fill=black] (7/8,13/16) circle [radius=0.005];
\draw[black,fill=black] (29/32,7/8) circle [radius=0.005];
\draw[black,fill=black] (15/16,15/16) circle [radius=0.005];

\draw[black,fill=black] (7/8,0) circle [radius=0.005];
\draw node[below] at (7/8,0) {\tiny$\frac{7}{8}$};
\draw[black,fill=black] (29/32,0) circle [radius=0.005];
\draw node[below] at (29/32,0) {\tiny$\frac{29}{32}$};
\draw[black,fill=black] (15/16,0) circle [radius=0.005];
\draw node[below] at (15/16,0) {\tiny$\frac{15}{16}$};

\draw[black,fill=black] (0,13/16) circle [radius=0.005];
\draw node[left] at (0,13/16) {\tiny$13/16$};
\draw[black,fill=black] (0,7/8) circle [radius=0.005];
\draw node[left] at (0,7/8) {\tiny$7/8$};
\draw[black,fill=black] (0,15/16) circle [radius=0.005];
\draw node[left] at (0,15/16) {\tiny$15/16$};

\draw[line width=0.0625mm,opacity=0.5] (15/16,15/16) -- (31/32,61/64);
\draw[line width=0.0625mm,opacity=0.5] (31/32,61/64) -- (63/64,63/64);

\draw[black,fill=black] (31/32,61/64) circle [radius=0.00125];
\draw[black,fill=black] (125/128,31/32) circle [radius=0.00125];
\draw[black,fill=black] (63/64,63/64) circle [radius=0.00125];

\draw[black,fill=black] (31/32,0) circle [radius=0.00025];
\draw[black,fill=black] (63/64,0) circle [radius=0.00025];

\draw[black,fill=black] (0,61/64) circle [radius=0.00025];
\draw[black,fill=black] (0,63/64) circle [radius=0.00025];

\draw[line width=0.003125mm,opacity=0.25] (63/64,63/64) -- (127/128,253/256);
\draw[line width=0.003125mm,opacity=0.25] (127/128,253/256) -- (255/256,255/256);

\draw[black,fill=black] (127/128,253/256) circle [radius=0.0000125];
\draw[black,fill=black] (255/256,255/256) circle [radius=0.0000125];

\draw[black,fill=black] (1,1) circle [radius=0.0000125];
\end{tikzpicture}
}

\newcommand{\BoundedSupport}{
\begin{tikzpicture}[scale=2]

\draw[->] (-1,-1) -- (-1,3);
\draw[->] (-1,-1) -- (3,-1);

\draw[line width=0.25mm] (-1,-1) -- (0,0);
\draw[line width=0.25mm] (1,1) -- (3,3);
\draw[black,fill=black] (0,0) circle [radius=0.03];
\draw[black,fill=black] (1,1) circle [radius=0.03];
\draw[black,fill=black] (3,3) circle [radius=0.03];
\draw[black,fill=black] (-1,-1) circle [radius=0.03];
\draw node[below left] at (-1,-1) {$0$};

\draw[line width=0.25mm,opacity=0.75] (0,0) -- (1/2,1/4);
\draw[line width=0.25mm,opacity=0.75] (1/2,1/4) -- (3/4,3/4);

\draw[black,fill=black] (1/2,1/4) circle [radius=0.02];
\draw[black,fill=black] (5/8,1/2) circle [radius=0.02];
\draw[black,fill=black] (3/4,3/4) circle [radius=0.02];

\draw node[below] at (1,-1) {$\frac{1}{2}$};\draw[black,fill=black] (1,-1) circle [radius=0.03];
\draw[black,fill=black] (0,-1) circle [radius=0.03];
\draw node[below] at (0,-1) {$\frac{1}{4}$};
\draw node[anchor=east] at (-1,1) {$\frac{1}{2}$};
\draw[black,fill=black] (-1,1) circle [radius=0.03];
\draw node[anchor=east] at (-1,0) {$\frac{1}{4}$};
\draw[black,fill=black] (-1,0) circle [radius=0.03];

\draw[line width=0.125mm,opacity=0.625] (3/4,3/4) -- (7/8,13/16);
\draw[line width=0.125mm,opacity=0.625] (7/8,13/16) -- (15/16,15/16);

\draw[black,fill=black] (7/8,13/16) circle [radius=0.01];
\draw[black,fill=black] (29/32,7/8) circle [radius=0.01];
\draw[black,fill=black] (15/16,15/16) circle [radius=0.01];

\draw[line width=0.0625mm,opacity=0.5] (15/16,15/16) -- (31/32,61/64);
\draw[line width=0.0625mm,opacity=0.5] (31/32,61/64) -- (63/64,63/64);

\draw[black,fill=black] (31/32,61/64) circle [radius=0.00125];
\draw[black,fill=black] (125/128,31/32) circle [radius=0.00125];
\draw[black,fill=black] (63/64,63/64) circle [radius=0.00125];

\draw[black,fill=black] (31/32,0) circle [radius=0.00025];
\draw[black,fill=black] (63/64,0) circle [radius=0.00025];

\draw[black,fill=black] (0,61/64) circle [radius=0.00025];
\draw[black,fill=black] (0,63/64) circle [radius=0.00025];

\draw[line width=0.003125mm,opacity=0.25] (63/64,63/64) -- (127/128,253/256);
\draw[line width=0.003125mm,opacity=0.25] (127/128,253/256) -- (255/256,255/256);

\draw[black,fill=black] (127/128,253/256) circle [radius=0.0000125];
\draw[black,fill=black] (255/256,255/256) circle [radius=0.0000125];

\draw[black,fill=black] (1,1) circle [radius=0.0000125];
\end{tikzpicture}
}

\newcommand{\LocalAction}{
\begin{tikzpicture}
\draw[fill=black] (0,0) circle [radius=0.05];
\draw[thick] (0,0) -- (-0.5,0.5);
\draw[thick] (0,0) -- (0.5,0.5);
\draw[fill=black] (-0.5,0.5) circle [radius=0.05];
\draw[thick] (-0.5,0.5) -- (-1,1);
\draw[thick] (-0.5,0.5) -- (0,1);
\draw[fill=black] (-1,1) circle [radius=0.05];
\draw[fill=black] (-0.5,1.5) circle [radius=0.05];
\draw[fill=black] (0.5,1.5) circle [radius=0.05];
\draw[thick] (0,1) -- (-0.5,1.5);
\draw[thick] (0,1) -- (0.5,1.5);
\draw[fill=black] (0.5,0.5) circle [radius=0.05];
\draw[black,fill=white,thick] (0,0) circle [radius=0.25] node {$a$};
\draw[black,fill=white,thick] (-0.5,0.5) circle [radius=0.25] node {$b$};
\draw[black,fill=white,thick] (0,1) circle [radius=0.25] node {$a$};
\end{tikzpicture}
}

\newcommand{\germ}{
\begin{tikzpicture}
\draw[fill=black] (0,0) circle [radius=0.05];
\draw[thick] (0,0) -- (-0.5,0.5);
\draw[thick] (0,0) -- (0.5,0.5);
\draw[thick] (-0.5,0.5) -- (-1,1);
\draw[thick] (-0.5,0.5) -- (0,1);
\draw[fill=black] (-1,1) circle [radius=0.05];
\draw[fill=black] (-0.5,1.5) circle [radius=0.05];
\draw[fill=black] (0.5,1.5) circle [radius=0.05];
\draw[thick] (0,1) -- (-0.5,1.5);
\draw[thick] (0,1) -- (0.5,1.5);
\draw[fill=black] (0.5,0.5) circle [radius=0.05];
\draw[black,fill=white,thick] (0,0) circle [radius=0.25] node {$a$};
\draw[black,fill=white,thick] (-0.5,0.5) circle [radius=0.25] node {$b$};
\draw[black,fill=white,thick] (0,1) circle [radius=0.25] node {$b$};
\end{tikzpicture}
}

\newcommand{\GroupElementA}{
	\begin{tikzpicture}
		\draw[thick] (0,0) -- (0.7,0.75);
		\draw[thick] (0,0) -- (-0.7,0.75);
		\draw[black,fill=white,thick] (0,0) circle [radius=0.25] node {$a$};
		\draw[fill=black] (-0.7,0.75) circle [radius=0.06];
		\draw[fill=black] (0.7,0.75) circle [radius=0.06];
		\draw[thick] (-0.7,0.75) -- (0.7,1.75);
		\draw[thick] (0.7,0.75) -- (-0.7,1.75);
		\draw[thick] (0.7,1.75) -- (0,2.5);
		\draw[thick] (-0.7,1.75) -- (0,2.5);
		\draw[black,fill=white,thick] (0,2.5) circle [radius=0.25] node {$b$};
		\draw[fill=black] (0.7,1.75) circle [radius=0.06];
		\draw[fill=black] (-0.7,1.75) circle [radius=0.06];
	\end{tikzpicture}
}
\newcommand{\GroupElementB}{
	\begin{tikzpicture}
		\draw[thick] (0,0) -- (0.6,0.65);
		\draw[thick] (0,0) -- (-1.2,1.25);
		\draw[thick] (0.6,0.65) -- (1.2,1.25);
		\draw[thick] (0.6,0.65) -- (0,1.25);
		\draw[black,fill=white,thick] (0,0) circle [radius=0.25] node {$a$};
		\draw[black,fill=white,thick] (0.6,0.65) circle [radius=0.25] node {$a$};
		\draw[fill=black] (1.2,1.25) circle [radius=0.06];
		\draw[fill=black] (0,1.25) circle [radius=0.06];
		\draw[fill=black] (-1.2,1.25) circle [radius=0.06];
		\draw[thick] (-1.2,1.25) -- (1.2,2.25);
		\draw[thick] (0,1.25) -- (-1.2,2.25);
		\draw[thick] (1.2,1.25) -- (0,2.25);
		\draw[thick] (0,3.5) -- (-0.6,2.85);
		\draw[thick] (0,3.5) -- (1.2,2.25);
		\draw[thick] (-0.6,2.85) -- (0,2.25);
		\draw[thick] (-0.6,2.85) -- (-1.2,2.25);
		\draw[black,fill=white,thick] (-0.6,2.9) circle [radius=0.25] node {$a$};
		\draw[black,fill=white,thick] (0,3.5) circle [radius=0.25] node {$b$};
		\draw[fill=black] (1.2,2.25) circle [radius=0.06];
		\draw[fill=black] (0,2.25) circle [radius=0.06];
		\draw[fill=black] (-1.2,2.25) circle [radius=0.06];
	\end{tikzpicture}
}
\newcommand{\GroupElementC}{
	\begin{tikzpicture}
		\draw[thick] (0,0) -- (1.2,1.25);
		\draw[thick] (0,0) -- (-0.6,0.65);
		\draw[thick] (-0.6,0.65) -- (-1.2,1.25);
		\draw[thick] (-0.6,0.65) -- (0,1.25);
		\draw[black,fill=white,thick] (0,0) circle [radius=0.25] node {$b$};
		\draw[black,fill=white,thick] (-0.6,0.65) circle [radius=0.25] node {$b$};
		\draw[fill=black] (1.2,1.25) circle [radius=0.06];
		\draw[fill=black] (0,1.25) circle [radius=0.06];
		\draw[fill=black] (-1.2,1.25) circle [radius=0.06];
		\draw[thick] (-1.2,1.25) -- (1.2,2.25);
		\draw[thick] (0,1.25) -- (-1.2,2.25);
		\draw[thick] (1.2,1.25) -- (0,2.25);
		\draw[thick] (0,3.5) -- (-0.6,2.85);
		\draw[thick] (0,3.5) -- (1.2,2.25);
		\draw[thick] (-0.6,2.85) -- (0,2.25);
		\draw[thick] (-0.6,2.85) -- (-1.2,2.25);
		\draw[black,fill=white,thick] (-0.6,2.85) circle [radius=0.25] node {$a$};
		\draw[black,fill=white,thick] (0,3.5) circle [radius=0.25] node {$b$};
		\draw[fill=black] (1.2,2.25) circle [radius=0.06];
		\draw[fill=black] (0,2.25) circle [radius=0.06];
		\draw[fill=black] (-1.2,2.25) circle [radius=0.06];
	\end{tikzpicture}
}

\providecommand{\keywords}[1]{\tbf{\textit{Index terms---}} #1}

\begin{document}

\title[Finitely presented simple groups with no piecewise projective actions]{Finitely presented simple groups with no piecewise projective actions}

\thanks{AB and RS are supported by the Australian Research Council Grant DP200100067.\\
RS is supported by an Australian Government Research Training Program (RTP) Scholarship.}
\author{Arnaud Brothier and Ryan Seelig}
\address{Arnaud Brothier, University of Trieste, Department of Mathematics, via Valerio 12/1, 34127, Trieste, Italy and
School of Mathematics and Statistics, University of New South Wales, Sydney NSW 2052, Australia}
\email{arnaud.brothier@gmail.com\endgraf
\url{https://sites.google.com/site/arnaudbrothier/}}
\address{Ryan Seelig, School of Mathematics and Statistics, University of New South Wales, Sydney NSW 2052, Australia}
\email{ryanseelig1997@gmail.com\endgraf
\url{https://sites.google.com/view/ryanseelig/}}


\begin{abstract} 
We construct an explicit infinite family of pairwise non-isomorphic infinite simple groups of type $\mathrm{F}_\infty$ (in particular, they are finitely presented) that act faithfully on the circle by orientation-preserving homeomorphisms, but that admit no non-trivial piecewise affine nor piecewise projective actions on the projective line.
Our examples are certain \emph{forest-skein groups} which, informally, are a mixture of Richard Thompson's groups with Vaughan Jones' planar algebras. 
\end{abstract}

\maketitle


\section*{Introduction}\label{sec:intro}

There has been much recent interest in constructing (infinite) simple groups witnessing unforeseen properties, see for instance \cite{Juschenko-Monod13,Hyde-Lodha19,Skipper-Witzel-Zaremsky19,Belk-Bleak-Matucci-Zaremsky}. Of particular interest are \emph{finitely presented} simple groups and a large source of such examples arise as ``Thompson-like'' groups acting on the circle by piecewise affine homeomorphisms having finitely many breakpoints, \cite{Brown87,Stein92,Lodha19,Burillo-Nucinkis-Reeves22}. However, among \emph{all} the examples of finitely presented simple groups acting on the circle in the literature, only one group $\Lambda$, constructed by Lodha, was known not to admit any piecewise affine action, \cite{Lodha19}. Lodha's group $\Lambda$ satisfies a slightly weaker property than acting piecewise affine: it acts by piecewise \emph{projective} homeomorphisms on the projective line. It is natural then to wonder if there are finitely presented simple groups acting on the circle, but not admitting any piecewise \emph{projective} actions.
Our main result is the following.

\begin{named}{Main Theorem}\label{theo:main}
There exist infinitely many explicit non-isomorphic simple groups having topological finiteness property $\mathrm{F}_\infty$ (in particular are finitely presented) which act faithfully on the circle by orientation-preserving homeomorphisms, but admit no non-trivial piecewise projective, hence no piecewise affine, action on the projective line.
\end{named}

We prove this in subsection \ref{sec:main-theorem}. As far as we're aware, our examples are the first groups of their kind. The word ``explicit'' is used to emphasise that our result is constructive --- our examples form a specific family (defined at the beginning of section \ref{sec:examples}) of \emph{forest-skein groups}.
Forest-skein groups were introduced in \cite{Brothier22} by the first author to strengthen deep connections uncovered by Jones between subfactor theory, conformal field theory, and Richard Thompson's groups $F,T,V$, \cite{Jones17,Brothier20}. Jones associated to any subfactor an algebraic quantum field theory with Thompson's group $T$ arising as its global symmetries. The core idea is then that a well-chosen forest-skein group will better reflect the structure of the subfactor and should thus replace $T$ as the global symmetries of the field theory, see the introduction of \cite{Brothier22} for further details. 
Forest-skein groups differ from most ``Thompson-like'' groups in that they are defined \emph{diagramatically}\footnote{Despite being defined with similar looking planar diagrams, forest-skein groups are not generically Guba--Sapir \emph{diagram} groups: the abelianisation of the former often has torsion, \cite[Sec.~3.6.3]{Brothier22}, while that of the latter is always free, \cite[Thm.~9.9]{Guba-Sapir06}.}, rather than \emph{dynamically}. This diagrammatic formalism permits to very easily construct interesting infinite discrete groups as well as giving powerful tools to study them, while it still has the merit to produce exotic dynamics like those described in the \ref{theo:main}.
Proving simplicity of the derived subgroup of a forest-skein group reduces (via classic techniques of Higman \cite{Higman54} and Epstein \cite{Epstein70}) to proving that its so-called canonical action is \emph{faithful}. 
In fact, the derived subgroup of the forest-skein group is simple \emph{if and only if} its canonical action is faithful, see section \ref{sec:simple} for details. Though, in practice, deciding faithfulness of the canonical action of a give forest-skein group is a difficult task. The examples in this article are similar to Lodha's group $\Lambda$ in that they also also satisfy a weaker property than being piecewise affine. However, ours go in a different direction: they are ``piecewise affine'' if we allow for having \emph{infinitely many} breakpoints in a ``controlled'' way, see subsections \ref{sec:graph} and \ref{sec:FGE} for  details.

{\bf Plan of the article.} This article is a shorter, self-contained version of a previous preprint \cite{Brothier-Seelig24b}. This preprint had as its goal to develop the general theory of forest-skein groups and in particular to explore their canonical dynamics. This general study lead to our actual \ref{theo:main} as an unanticipated corollary. We should also note that the general theory produces many other examples of FS groups witnessing the \ref{theo:main}.

This article is split into two sections. In section \ref{sec:prelim} we give a brief introduction to forest-skein groups and explain some key properties they enjoy. The reader familiar with the binary tree calculus of Thompson's groups $F,T,V$ (and some of their cousins) may like to skip section \ref{sec:prelim} and refer back to it when needed. An intuitive picture of forest-skein groups may then be helpful: they consist of tree-pair diagrams having coloured non-leaf vertices which are taken equivalent up to changing certain patterns, similar (in an informal sense) to the tree-pair diagrams for Brin's group \cite{Brin04,Burillo-Cleary10}, Cleary's group \cite{Cleary00,Burillo-Nucinkis-Reeves21}, the Lodha--Moore groups \cite{Lodha-Moore16}, and Stein's groups \cite{Stein92}, see also example \ref{ex:fraction-trees}. 

In the section \ref{sec:examples} we define a very specific family of forest-skein groups. We show that a powerful reconstruction theorem of McCleary and Rubin applies to our examples, which implies each of them encodes a faithful action on the circle. This action is an isomorphism invariant and this fact permits us to distinguish all of our examples. We show these examples witness the \ref{theo:main} and then finish the article with a comparison of our groups to some others.

\subsection*{Acknowledgements} We warmly thank Matt Brin, Yash Lodha, Matt Zaremsky, Christian de Nicola Larsen, and Dilshan Wijesena for their constructive and pertinent feedbacks on a previous iteration of this article. We are very grateful for the anonymous referees who drastically helped us improving the quality of this article. The second author thanks the Australian Research Council and the University of Trieste which, respectively, supported and hosted him during parts of the writing of this article.

\subsection*{Notations and conventions} Here are some common notations and conventions that we use throughout the article that the reader may like to refer back to.

{\bf General notation.}

\begin{itemize}
    \item $\R$ is the real line, $\bS$ is the circle, $\bP$ is the projective line, $\fC$ is the Cantor space.
\item $D(\Ga)$ is the derived subgroup of a group $\Ga$ and $\Ga^{ab}:=\Ga/D(\Ga)$ is the abelianisation.
\item $\Gr\la S|R\ra$ is the group defined by generators $S$ and relations $R$.
\item $\Z_n$ is the cyclic group of order $n$.
\item $\Homeo(X)$ is the homeomorphism group of a topological space $X$.
\item If $\Ga\act X$ is a \emph{faithful} group action by homeomorphisms, we may identify $\Ga$ as a subgroup of $\Homeo(X)$ and write $g(x)$ for the action of $g\in\Ga$ on $x\in X$.
\item We say a function $f:X\to X$ \emph{fixes} a subset $A\subset X$ if $f(a)=a$ for all $a\in A$.
\end{itemize}
{\bf Forest-skein notation.}
\begin{itemize}
\item FS stands for ``forest-skein''.
\item We use symbols $r,s,t$ for trees and $f,g,h$ for forests and use the notation $t(a)$ for the coloured tree with shape $t$ whose every interior vertex is coloured by $a$.
\item $\cF$ is an Ore FS category, $L,G,M$ are its $F,T,V$-type FS groups. Elements of $L$ (resp.~$G,M$) are written $[t/s]$ (resp.~$[t/\pi/s]$) with $t,s$ trees and $\pi$ a permutation.
    \item $\FS\la S|R\ra$ is the FS category defined by colours $S$ and skein relations $R$.
    We often have $S=\{a,b\}$ and $R$ being a single skein relation $x(a) = y(b).$
\item $\FS\la S\ra$ is the free FS category with colours $S$ (and no skein relations).
\item $\lambda_n$ (resp.~$\rho_n)$ is a monochromatic left- (resp.~right-) vine with $n$ carets.
\item $\tau_n:=Y\circ(\lambda_n\ot I)$ for all $n\geqslant3$, $\cF_n:=\FS\la a,b|\tau_n(a)=\rho_n(b)\ra$, and $L_n,G_n,M_n$, are the $F,T,V$-type FS groups of $\cF_n$.
\end{itemize}

\section{Preliminaries}\label{sec:prelim}

In this section we give a brief but completely self contained overview of forest-skein groups (henceforth FS groups). 
FS groups are constructed by taking fractions of morphisms in a small category that has a diagramatic calculus. The underlying category is used extensively to study the FS groups produced, however, one needs no prior category theory knowledge in order to start working with FS groups. 
After defining FS categories and groups we discuss some properties they enjoy that are needed in this article. We invite the interested reader to consult \cite{Brothier22} for further details on the general theory.

\subsection{Forest-skein categories}{\bf Words.} Given a non-empty set $A$ the set $A^*$ of all \emph{words over $A$} is consists of all ordered tuples having entries in $A$, that is, $A^*=\cup_{n\geqslant0} A^n$. We write $a_1a_2\cdots a_n$ instead of $(a_1,a_2,\dots,a_n)$ and denote the unique element of $A^0$ by $\varnothing$, which is called the \emph{empty word}. 
The set of words over $A$ forms a monoid under concatenation: $(a_1\cdots a_n)\cdot (b_1\cdots b_m):=a_1\cdots a_nb_1\cdots b_m$, which has neutral element $\varnothing$. Sometimes we may supress the $\cdot$ and write $u\cdot v=uv$. For any $w\in A^*$ denote by $w^n$ the word that is $w$ concatenated with itself $n$ times. We say $u\in A^*$ is a \emph{subword} of $w\in A^*$ if $w=w'\cdot u\cdot w''$ for some (possibly empty) words $w',w''\in A^*$.

{\bf Sequences.} Let $A^\omega$ be the set of all (infinite) sequences having entries in $A$. We may write $a_1a_2a_3\cdots$ instead of $(a_1,a_2,a_3,\dots)$. We often treat $A^\omega$ as a left $A^*$-set with action $a_1\cdots a_n\cdot b_1b_2\dots:=a_1,\cdots a_nb_1b_2\cdots$.  Given $w\in A^*$ we write $w\cdot A^\omega$ for the set $\{w\cdot x:x\in A^\omega\}$ of all sequences having \emph{prefix} $w$ and we also denote by $\ov w$ the infinite sequence $w\cdot w\cdot w\cdots$, that is, $w$ concatenated with itself infinitely many times.

{\bf Trees.} The \emph{infinite binary tree} $\fT$ is the graph having vertex set $\{\0,\1\}^*$ and directed edges $(w,w\cdot\0)$ and $(w,w\cdot\1)$ for each $w\in\{\0,\1\}^*$. The vertices $w\cdot\0$ and $w\cdot\1$ are the \emph{left} and \emph{right children} of $w$, respectively, and the empty word $\varnothing$ is called the \emph{root} of $\fT$. In this article a \emph{tree} $t$ is a finite connected subgraph of $\fT$ containing the root $\varnothing$ such that for each vertex in $t$ either both children are also contained in $t$ or neither of them are (i.e.,~these are the finite full rooted subtrees of $\fT$). Let $t$ be a tree. If $\nu$ is a vertex of $t$ such that $\nu\cdot\0$ and $\nu\cdot\1$ are also vertices of $t$, then $\nu$ is an \emph{interior vertex of $t$}. Otherwise, $\nu$ is a \emph{leaf of $t$}, and we write $\Leaf(t)$ for the set of all leaves of $t$. The leaves of $t$ are totally ordered using the lexicographical order on $\{\0,\1\}^*$ with $\0<\1$. Hence, we often identify the leaf set of a tree $t$ having $n$ leaves with $\{1,2,\dots,n\}$ with its usual order. We may then say the word in $\{\0,\1\}^*$ associated to the $i$th leaf of $t$ is the \emph{address of the $i$th leaf.}

{\bf Paths.} Let $u$ and $v$ be (not necessarily interior) vertices in a tree $t$. A \emph{path from $u$ to $v$} in $t$ is a finite sequence of vertices $(\nu_1,\nu_2\dots,\nu_l)$ in $t$ such that $\nu_1=u$, $\nu_l=v$, and $ (\nu_i,\nu_{i+1})$ is a directed edge in $t$ for all $1\leqslant i<l$. Note that in any tree $t$ there is a \emph{unique} path from its root to any one of its leaves.

{\bf Special trees.} The unique tree having one leaf (and no interior vertices) is called the \emph{trivial tree} and is denoted $I$, while the unique tree having two leaves (and one interior vertex) is called the \emph{caret} and is denoted $Y$. 

{\bf Subtrees.} Let $t$ be a tree and $w$ an interior vertex of $t$. A tree $s$ is called a \emph{subtree of $t$ rooted at $w$} if for every interior vertex $\nu$ of $s$, we have that $w\cdot\nu$ is an interior vertex of $t$. If $w=\varnothing$, then we say $s$ is a \emph{rooted subtree of $t$}. In general a tree $s$ will not be a subtree of a tree $t$ in a unique way. Indeed, for any tree $t$, the caret $Y$ is a subtree rooted at each interior vertex of $t$ and we call such a subtree a \emph{caret of $t$}. Note that carets of $t$ are in one-to-one correspondence with interior vertices of $t$.

{\bf Forests.} A \emph{forest} $f$ is an ordered tuple of trees $(f_1,f_2,\dots,f_n)$. We call $n$ the \emph{number of roots of $f$} and we identify a trees as forests with exactly $1$ root. A forest is \emph{trivial} if all of its trees are trivial, so there is exactly one trivial forest for each $n\geqslant1$. 

{\bf Coloured forests.} Let $S$ be a non-empty set. A \emph{tree coloured over $S$} is a tree $t$ having each interior vertex (i.e.,~a non-leaf vertex) decorated by an element of $S$ and a \emph{forest coloured over $S$} is an ordered tuple of coloured trees over $S$. If $S$ is clear from context we may drop ``coloured over $S$'' and refer them as forests. We emphasise that forests with no colouring are \emph{monochromatic}. For $a\in S$, and a tree $t$ coloured over $S$, we may call an $a$-coloured vertex in a tree an \emph{$a$-vertex}. Given a monochromatic tree $t$ and a colour $a$ we write $t(a)$ for the coloured tree having only $a$-vertices. The \emph{shape} of $t(a)$ is $t$ and generally we use the word shape to express the underlying monochromatic tree of a coloured tree.

{\bf Coloured carets.} As previously mentioned a caret is referred to as either a tree or a subtree with shape $Y$. If $a\in S$ is a colour, then we may call an \emph{$a$-caret} or \emph{caret coloured by $a$} the tree or subtree with shape $Y$ and whose single interior vertex is coloured by $a$. Since we deal only with binary trees, each tree having $n$ leaves has exactly $n-1$ carets.
 
{\bf Diagrams for coloured forests.} We can identify trees with isotopy classes of their geometric realisations inside the strip $\R\times[0,1]$ which have their leaves contained in order in $\R\times\{1\}$ and all interior vertices contained in $\R\times(0,1)$. Technically, we should also connect the root of a tree to a point in $\R\times\{0\}$ by a line segment. We suppress drawing this for non-trivial trees, but draw it for trivial trees. We call such an isotopy class a \emph{tree diagram.} If the tree is coloured, then we decorate each interior vertex in its tree diagram with its respective colour. Similarly, we identify a forest $f=(f_1,f_2,\dots,f_n)$ with its \emph{forest diagram} in $\R\times[0,1]$ which consists of the union of the tree diagrams of its trees drawn so the root and all leaves of the $i$th tree $f_i$ to the left of the root and all leaves of the $j$th tree $f_j$ for all $i<j$. If the forest is coloured we decorate all interior vertices of its trees with the respective colourings. In practice we draw leaves anywhere above the interior vertices in the right order (so not necessarily having height $1$). For example, a forest coloured over $\{a,b\}$ having $3$ roots and $6$ leaves is:
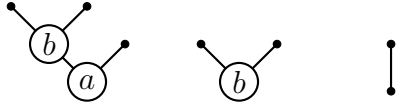
\begin{figure}[H]
\centering
\FigD
\end{figure}
The third tree in this forest is the trivial tree $I$.
We now introduce two natural binary operations on forests coloured over a fixed set $S$.

{\bf Composition.} Let $f$ be a forest with $l$ roots and $m$ leaves and let $g$ be a forest with $m$ roots and $n$ leaves. We may form their \emph{composition} $f\circ g$ as follows. 
Stack the diagram for $g$ on top of the diagram for $f$ and identify the $i$th leaf of $f$ with the $i$th root of $g$ for all $1\leqslant i\leqslant m$. After scaling $\R\times[0,2]$ affinely to $\R\times[0,1]$ this yields a new forest-diagram. We equip this forest-diagram with the natural colouring inherited from $f$ and $g$ (this is well-defined as leaves are not coloured). For example:
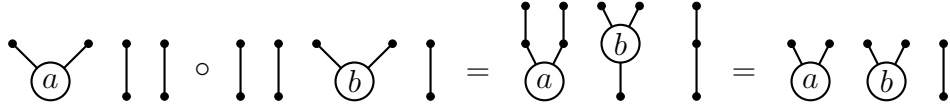
\begin{align*}
\Composition
\end{align*}
Composition is associative (where defined) and has trivial forests as neutral elements. We will often write $fg$ instead of $f\circ g$ usually without qualifying the number of roots and leaves of $f$ and $g$, leaving implicit that the number of leaves $f$ must equal to the number of roots of $g$. We may also refer to the process of going from $f$ to $fg$ as \emph{growing $f$ by $g$}.

{\bf Tensor product.} The \emph{tensor product} of any two forests as $(f_1,\dots,f_n)\ot(g_1,\dots,g_m):=(f_1,\dots,f_n,g_1,\dots,g_m)$. Diagrammatically this is horizontal concatenation, for example:

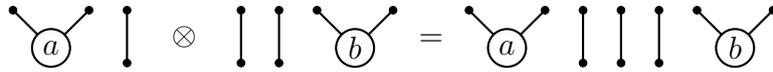
\begin{figure}[H]
\centering
\TensorProd
\end{figure}

{\bf Notations for some special forests.} The monochromatic \emph{left-vine} and \emph{right-vine} with $n$ carets are defined by
\begin{align*}
\lambda_n:=Y(Y\ot I)\cdots(Y\ot I^{\ot n-1})\quad\textnormal{and}\quad\rho_n:=Y(I\ot Y)\cdots(I^{\ot n-1}\ot Y),
\end{align*}
respectively. For example
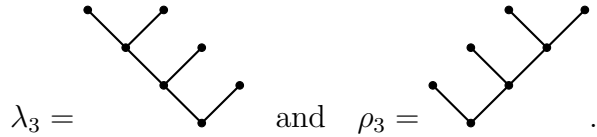
\begin{align*}
\lambda_3=\LeftVine\quad\textnormal{and}\quad\rho_3=\RightVine.
\end{align*}
We may refer to a coloured tree whose shape is a left or right-vine as a left or right-vine, respectively. As previously mentioned, given a monochromatic tree $t$ and a colour $a$ we write $t(a)$ for the tree $t$ in which every interior vertex is coloured by $a$. Such a tree may be called \emph{an $a$-tree} and a forest made of $a$-trees is an \emph{$a$-forest}.

{\bf Free forest-skein category.} Let $S$ be a non-empty set and write $\FS\la S\ra$ for the set of all forests coloured over $S$. If we set the source and target of a forest to be its number of leaves and roots respectively, then $\FS\la S\ra$ is the morphism set of a small category whose objects are the natural numbers $\N$ and whose composition is composition of forests. By abuse of notation we denote this category by $\FS\la S\ra$ and we call it the \emph{free forest-skein category over $S$}. We also adopt this convention for all other categories appearing in this article (that is, viewing it as a set with a partially defined associative binary operation, as is commonly done for groupoids). Up to adding a formal ``empty forest'' having $0$ roots and $0$ leaves to $\FS\la S\ra$, the tensor product turns $\FS\la S\ra$ into a (strict) monoidal category in the sense of \cite[Chapter VII]{Maclane78}, that is, $\ot:\FS\la S\ra\times\FS\la S\ra\to\FS\la S\ra$ is an associative bifunctor having empty forest as neutral element, defined on objects as $n\ot m:=n+m$, and morphisms as the tensor product of forests.

{\bf Skein relations and forest-skein categories.} A \emph{skein relation over $S$} is a pair of trees $(u,v)$ (often written $u=v$ or $u\sim v$) coloured over $S$ and having the same number of leaves, or equivalently, the same number of carets. For example:
\begin{align}\label{fig:skein-example}
\SkeinRelationThree    
\end{align}
is a skein relation over $\{a,b\}$. Given a set $R$ of skein relations over $S$ we write $\ov{R}$ for the smallest equivalence relation containing $R$ on $\FS\la S\ra$ satisfying for all $(x,y)\in\ov R$ we have
\begin{align*}
(f\circ x\circ g,f\circ y\circ g)\in\ov R\quad\textnormal{and}\quad(p\ot x\ot q,p\ot y\ot q)\in\ov R
\end{align*}
for all $f,g,p,q\in\FS\la S\ra$. These conditions ensure that the quotient set $\FS\la S\ra/\ov{R}$ induces a monoidal category structure with objects $\N$, composition $[f]\circ[g]:=[f\circ g]$, and tensor product $[p]\ot[q]:=[p\ot q]$. Note that if trees in skein relations did not have the same number of leaves, then $\circ$ would not pass to the quotient. 

\begin{definition}
We call $\FS\la S|R\ra=\FS\la S\ra/\ov R$ the \emph{forest-skein (FS) category with skein presentation $\la S|R\ra$}.
\end{definition}

We will always denote the class $[f]$ by $f$ and to avoid confusions we distinguish forests in $\FS\la S\ra$ from forests in $\FS\la S|R\ra$ by insisting the former are \emph{free forests} (meaning no skein relations involved). We typically denote an arbitrary FS category by $\cF$ leaving implicit an underlying skein presentation.

{\bf Diagrams for permutations.} We also want to allow permutations of the leaves of our forests. Fix $n\geqslant1$ and a permutation $\pi:\{1,2,\dots,n\}\to\{1,2,\dots,n\}$. We identify $\pi$ with the diagram in $\R\times[0,1]$ consisting of the union of line segments going from $(1,i)$ to $(0,\pi(i))$ for each $1\leqslant i\leqslant n$. Each such line segment is called a \emph{strand of $\pi$}. For $\sigma$ and $\tau$ both having $n$ strands we form $\sigma\circ\tau$ by taking the union of all line segments from $(1,i)$ to $(0,\sigma(\tau(i)))$ for $1\leqslant i\leqslant n$. This operation can be interpreted diagrammatically as vertical stacking. For example
\begin{align*}
(12)\circ(123)\quad=\quad\PermutationA\quad=\quad\PermutationB\quad=(23).
\end{align*}
With this operation, the set all such diagrams with a fixed number $n$ many strands forms a group $\Sigma_n$ which is isomorphic to the symmetric group on $n$ letters. 
The identity is the diagram consisting of vertical strands and the inverse of $\pi$ is the reflection of $\pi$ about the line $y=1/2$. A permutation with $n$ strands is \emph{cyclic} if it is of the form $i\mapsto i+j~\textnormal{mod}~n$ for some fixed $j$. In cycle notation a permutation is cyclic if and only if it is in the subgroup generated by $(12\cdots n)$ which is isomorphic to the cyclic group $\Z_n$, justifying the terminology. However, we warn the reader a cycle with our definition may have more than one orbit, so is not cyclic in that sense, e.g.,~$(1234)\circ(1234)=(13)(24)$.
One can define $\Sigma_n$ using \emph{classes} of diagrams similar to Artin braid groups though allowing only one type of crossing rather than two, see \cite{Artin47}.

{\bf Symmetric forests.} Now let $\cF$ be an FS category. We may consider forests in $\cF$ having a permutation with as many strands as there are leaves of said forest stacked on top. We may write such a diagram formally as $f\circ\pi$, where $f\in\cF$ has $n$ leaves and $\pi$ is a permutation on $n$ letters. We may call such a diagram a \emph{symmetric forest}. We would like to extend the category structure of $\cF$ to symmetric forests. To compose two symmetric forests $f\circ\pi$ and $g\circ\sigma$ we want a way of writing $\pi\circ g$ as $g^\pi\circ\pi^g$ for some forest $g^\pi\in\cF$ and permutation $\pi^g$. This will enable us to extend composition to symmetric forests as
\begin{align*}
(f\circ\pi)\circ(g\circ\sigma)=fg^\pi\circ\pi^g\sigma.
\end{align*}
We define for each forest $g$ having $m$ roots and $n$ leaves and each permutation $\pi$ having $m$ strands
\begin{itemize}
    \item the forest $g^\pi$ whose $i$th tree is the $\pi(i)$th tree of $g$; and
    \item a permutation $\pi^g$ having $n$ strands formed by splitting the $i$th strand of $\pi$ into $|\Leaf(g_i)|$ many parallel strands for all $1\leqslant i\leqslant m$.
\end{itemize}
Note, this is a Brin-Zappa-Sz\'ep product (or bicrossed product), see \cite{Brin05}.
These definitions become transparent once we draw our forests.
The equality $\pi\circ g=g^\pi\circ \pi^g$ can be visualised by ``dragging'' the trees of $g$ down to the bottom of the diagram via the strands of $\pi$. For example:
\begin{align*}
\pi\circ g=\vcenter{\hbox{\Cabelling}}=g^\pi\circ\pi^g.
\end{align*}
Here we have
\begin{align*}
\pi=(12)\in\Sigma_3,~~g=Y_a\ot I\ot I,~~g^\pi=I\ot Y_a\ot I,~~\textnormal{and}~~\pi^g=(123)\in\Sigma_4.
\end{align*}
We identify the trivial permutation with $n$ strands with the trivial forest $I^{\ot n}$ with $n$ roots. 
By definition it follows that $(I^{\ot n})^g=I^{\ot m}$ if $g$ is a forest with $n$ roots and $m$ leaves. As a diagram the cyclic permutation $(12\cdots n)^m$, $1\leqslant m<n$, consists of $n-m$ strands going down to the right followed by $m$ strands going down to the left. Hence, if $\pi$ is cyclic, then $\pi^g$ is also cyclic for any forest $g$.
Once again one could as well define this algebraic structure using directly classes of diagrams (similarly than Artin braid groups and variants of it).

{\bf Extended FS categories.} Permutation diagrams permit us to extend FS categories into the associated \emph{$T$-type FS category} $\cF^T$ and the associated \emph{$V$-type FS category} $\cF^V$ (having the same objects as $\cF$) by adding a new isomorphism $\pi$ of the object $n$ to $\cF$ for every $n$-strand cyclic permutation or arbitrary permutation, respectively. That the above composition is associative where defined and unital (hence defining a category structure) is a straightforward exercise in unwrapping the definitions of $g^\pi$ and $\pi^g$. Hence, morphisms in $\cF^V$ are symmetric forests. Note the monoidal structure of $\cF$ extends to $\cF^V$ but not to $\cF^T$ as the tensor product (horizontal concatenation) of two cyclic permutations is no longer cyclic in general. We have containments as (wide) subcategories $\cF\subset\cF^T\subset\cF^V$. 

\subsection{Forest-skein groups} 
We now define the central objects of this article: FS groups.  We refer the reader to \cite[Section 3]{Brothier22} for further details. A (small) category $\cC$
\begin{itemize}
    \item is \emph{left-cancellative} if whenever there is an equality of morphisms $f\circ p=f\circ q$ of $\cC$, we have $p=q$; and
    \item has \emph{right-common-multiples} if for any morphisms $g,h$ of $\cC$ with the same target there are morphisms $k,l$ of $\cC$ such that $g\circ k=h\circ l$.
\end{itemize}
A category $\cC$ with these properties is an \emph{Ore category}. Note an FS category $\cF$ is Ore if and only if $\cF^T$ is Ore or if $\cF^V$ is Ore. This comes from the fact that the three categories $\cF,
\cF^T,\cF^V$ are equal up to their sets of isomorphisms. 
Given an Ore FS category $\cF^V$ we may construct its \emph{groupoid of fractions}. Let $\cD^V_{\cF}$ be the set of all triples $(f/\pi/g)$, where $f,g\in\cF$ are forests with the same number of leaves and $\pi$ is a permutation of their leaves. We think of $(f/\pi/g)$ as a diagram $f\circ\pi\circ g^{-1}$, where $g^{-1}$ denotes the diagram of $g$ reflected about $y=1/2$: hence a diagram with $f$ below, $\pi$ stacked in the middle, and the diagram of $g$ drawn upside down and stacked on top. 
A diagram is \emph{free} if $f$ and $g$ are free forests. Consider the equivalence $\sim$ on $\cD^V_\cF$ generated by ``growing relations''
\begin{align}\label{eq:grow}
(f/\pi/g)\sim(fp^\pi/\pi^p/gp)  
\end{align}
for all $(f/\pi/g)\in\cD^V_\cF$ and $p\in\cF$. This is coming from the equality of formal diagrams $f\pi g^{-1}=f\pi pp^{-1}g^{-1}=(fp^\pi)\pi^p(gp)^{-1}$. The class $[f/\pi/g]$ of $(f/\pi/g)$ modulo $\sim$ is called a \emph{fraction of forests} or a \emph{forest-pair diagram}.

\begin{example}\label{ex:fraction-trees}
Taking colour set $\{a,b\}$ and the skein relation 
\begin{align*}
Y_b(Y_b\ot I)=\Cleary=Y_a(I\ot Y_a)
\end{align*}
an example of a fraction of \emph{trees} (or tree-pair diagram) is given by the diagrams:
\begin{align}\label{fig:tree-diagram}
\GroupElementA
\quad\sim_{grow}\quad
\GroupElementB
\quad\sim_{skein}\quad
\GroupElementC
\end{align}
Here we have grown the triple $(Y_a/(12)/Y_b)$ into $(Y_a(I\ot Y_a)/(123)/Y_b(Y_a\ot I))$ using the relation \ref{eq:grow} where
\begin{align*}
p=Y_a\ot I=\ForestA~.
\end{align*}    
\end{example}

If $\pi$ is the identity permutation, we suppress the permutation and write $(f/g)$ and $[f/g]$ instead. Having left-cancellation and right-common-multiples in $\cF^V$ implies the formulae
\begin{align*}
[f/\pi/g]\cdot[g/\sigma/h]:=[f/\pi\sigma/g]\quad\textnormal{and}\quad[f/\pi/g]^{-1}:=[g/\pi^{-1}/f]
\end{align*}
determine a groupoid structure on the quotient $\Frac(\cF^V):=\cD^V_\cF/\sim$ that we call the \emph{fraction groupoid of $\cF^V$}. The source and target of $[f/\pi/g]$ are the number of roots of $g$ and the number of roots of $f$, respectively. Similarly, one constructs the fraction groupoids of $\cF$ and $\cF^T$ by considering fractions with no permutations or cyclic permutations, respectively. These groupoids naturally embed into $\Frac(\cF^V)$. Recall the \emph{isotropy group} of a groupoid $\cG$ at an object $x\in\cG$ is the group of all morphisms in $\cG$ having source and target equal to $x$.

\begin{definition}
Let $\cF$ be an Ore FS category. 
\begin{enumerate}
    \item The isotropy group of the fraction groupoid $\Frac(\cF^V)$ at the object $1$ is called the \emph{$V$-type FS group of $\cF$}. We denote this group $G^V_\cF$. 
    \item The isotropy group of the fraction groupoid $\Frac(\cF^T)$ at the object $1$ is called the \emph{$T$-type FS group of $\cF$}. We denote this group $G^T_\cF$.
    \item The isotropy group of the fraction groupoid $\Frac(\cF)$ at the object $1$ is called the \emph{$F$-type FS group of $\cF$}. We denote this group $G^F_\cF$. 
\end{enumerate}
Collectively, these are the \emph{FS groups of $\cF$}. 
\end{definition}

Unwrapping the definition of isotropy group at $1$, it follows that FS groups consist of fractions of \emph{trees} with a permutation between their leaves like in diagram \ref{fig:tree-diagram}. We always specify a family of FS groups with the single Ore FS category $\cF$ and leave implicit the construction of $\cF^T$ and $\cF^V$ strictly needed to build the $T$ and $V$-type FS groups.
There is also a braided version for FS groups similar to the braided Thompson group, \cite{Brin07,Dehornoy06}. We don't consider these in this article, see \cite[Sec.~1.6]{Brothier22} for details.

\subsection{Initial examples} 
Deciding if a given skein presentation produces an Ore FS category is a challenging problem. Fortunately, the first author has produced an extremely vast class of explicit examples, see \cite[Thm.~8.3]{Brothier22}.

\begin{theorem}\label{theo:Ore-categories}
Let $x$ and $y$ be any monochromatic trees with the same number of leaves and consider the FS category $\cF=\FS\la a,b|x(a)=y(b)\ra$ (Hence, there are two colours: $a$ and $b$, and one skein relation $x(a)=y(b)$. This means that $x(a)$ has shape $x$ with all interior vertices coloured by $a$ and within any forest the pattern $x(a)$ can be interchanged with $y(b)$, and vice versa).
Then, $\cF$ is left-cancellative. Moreover, any forest $f\in\cF$ can be grown (composed with some forest in $\cF$) into an $a$-coloured forest implying that $\cF$ has right-common-multiples. All in all, $\cF$ is an Ore FS category.
\end{theorem}
These FS categories and their groups are the focus for the remainder of this article. Let us briefly discuss some examples to get a sense of this class. Taking $x$ and $y$ both equal to a caret yields Thompson's groups $F,T,V$, \cite[Ex.~3.5]{Brothier22}. Taking $x$ to be the left-vine with $2$ carets and $y$ to be the right-vine with $2$ carets yields Cleary's irrational slope Thompson groups $F_\tau,T_\tau,V_\tau$, \cite[Sec.~3.6.3]{Brothier22}. These groups admit faithful action by piecewise affine \emph{bijections} on $[0,1]$ ($F_\tau$ acts by homeomorphisms but elements of $T_\tau$ and $V_\tau$ may have finitely many discontinuity points when acting on $[0,1]$). Moreover, $T_\tau$ and $V_\tau$ are virtually simple (having simple derived subgroup with finite index), \cite{Burillo-Nucinkis-Reeves22}. Taking $x$ and $y$ both left-vines yields FS groups that decompose as twisted permutational wreath products that are extensions of the Thompson groups,  \cite[Thm.~3.4]{Brothier23}. In contrast, these do not have a simple derived subgroup nor any obvious (non-Thompson group) simple subgroup. 

\subsection{Abelianisation} 

We now explain how ``counting colours'' gives a description of the abelianisation of $T$-type FS groups. Given a free tree $t$ over $a,b$, let $\chi(t)$ be the number of $b$-vertices in $t$.

\begin{theorem}\label{theo:abelianisation}
Consider the $T$-type FS group $G$ of $\FS\la a,b|x(a)=y(b)\ra$, where $x$ and $y$ are trees both having $n\geqslant1$ carets. Then
\begin{align*}
\ov\chi:G\to\Z_n,~[t/\pi/s]\mapsto\chi(t)-\chi(s)~\textnormal{mod}~n
\end{align*}
is a well-defined epimorphism and which induces an isomorphism on $G^{ab}=G/D(G)$.
\end{theorem}

Here is a quick sketch of the proof. Changing $(t/\pi/s)$ by a skein relation increases or decreases its number of $b$-vertices by $n$, so there is no effect modulo $n$. For any forest $f$ with the same number of roots as there are leaves of $t$ and $s$ we have that $f$ and $f^\pi$ have the same number of $b$-vertices. Hence, growing the triple $(t/\pi/s)$ into $(tf^\pi/\pi^f/sf)$ does not change $\chi(t)-\chi(s)$. It follows that $\ov\chi$ is a well-defined group morphism that is surjective as $\ov\chi([Y_b/Y_a])=1$ generates $\Z_n$.
Hence, $\ov\chi$ factors through an epimorphism $G^{ab}\onto\Z_n$. We show $G^{ab}\onto\Z_n$ is injective by showing $[Y_b/Y_a]^n\in D(G)$ using some relations that hold in $G$.
We refer the reader to \cite[Thm.~5.1]{Brothier-Seelig24a} for further details.  

\subsection{Finiteness properties}

Recall a group $\Gamma$ is of \emph{type $\mathrm{F}_\infty$} if it admits an Eilenberg-MacLane $K(\Gamma,1)$ complex with finite $n$-skeleton for all $n$. All $F$ and $V$-type FS groups can be interpreted as \emph{operad groups} of Thumann, see \cite[Sec.~3]{Thumann17}. Via a deep analysis, Thumann proved a very general topological finiteness result for operad groups, see \cite[Thm.~4.3]{Thumann17}. Hence, under a mild hypothesis on an Ore FS category $\cF$, Thumann's theorem applies and shows $G^F_\cF$ and $G^V_\cF$ are of type $\mathrm{F}_\infty$. One may then deduce the $T$-type group of such a category $\cF$ is type $\mathrm{F}_\infty$ by employing an argument of Brown and Geoghegan, see \cite[Thm.~6.4]{Brothier22} for a detailed explanation. In particular:

\begin{theorem}\cite[Cor.~7.5]{Brothier22} \label{theo:finiteness-properties}
For any monochromatic trees $x,y$ with the same number of leaves, the $F$, $T$, and $V$-type FS groups of $\FS\la a,b|x(a)=y(b)\ra$ are of type $\mathrm{F}_\infty$.
\end{theorem}

\subsection{Simplicity}\label{sec:simple}

The category $\cF$ underlying an $X$-type FS group, $X=F,T,V$, affords it an action $\al^X:G_\cF^X\act\fC_\cF$ on a profinite space (it is a Cantor space if $\cF$ is countable). This is the \emph{canonical action of $G_\cF^X$} and we will outline one basic construction of it for our main examples in subsection \ref{sec:dynamics}. One of the core results of \cite{Brothier-Seelig24a} is that this action controls the simplicity properties of $G^X_\cF$. In this article we need only the result for the $T$-type groups (and for a certain subgroup, see proposition \ref{prop:D(K)-simple}).

\begin{theorem}\cite[Thm.~3.7, Thm.~3.10]{Brothier-Seelig24a}\label{theo:simple}
Let $G$ be a $T$-type FS group and $\al:G\act\fC$ its canonical action. Then $D(G)$ is simple if and only if $\al$ is faithful.
\end{theorem}

The argument for the ``if'' direction is well known throughout the Thompson group literature and is based on classic ideas of Higman and Epstein, \cite{Higman54,Epstein70}. 
The ``only if'' direction makes heavy use of the diagrammatic calculus afforded by the underlying FS category. There are many examples of FS groups having non-faithful canonical action, see \cite[Ex.~4.4]{Brothier-Seelig24a}. Conversely, proving faithfulness of the canonical action of a given FS group is quite involved, see \cite[Sec.~6]{Brothier-Seelig24a} and subsection \ref{sec:faithfulness}.

\subsection{Two quotients of $F$-type FS groups}\label{sec:quotients}
Every $F$-type FS group admits two quotients analogous to the derivative maps for $F\act[0,1]$ at $0$ and $1$, see \cite[Sec.~3.2]{Brothier-Seelig24a}. 
We now briefly recall this construction. Fix a skein presentation $\la S|R\ra$ (recall $S$ is a set of colours and $R$ is a set of pairs of free trees over $S$) whose associated FS category $\cF=\FS\la S|R\ra$ is an Ore category and write $L$ for its $F$-type FS group. 
Given a \emph{free} forest $f$ over $S$ (i.e.,~an element of the free FS category $\FS\la S\ra$) we let $c^-(f)$ (resp.~$c^+(f)$) be the word over $S$ defined by reading colours along the path in the first (resp.~last) tree of $f$ from its root to its first leaf (resp.~last). For example, if
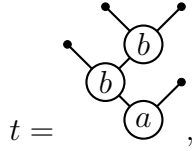
\begin{align*}
t=\germ,
\end{align*}
then $c^-(t)=ab$ and $c^+(t)=a$. By definition we have $c^\pm(f\circ g)=c^\pm(f)\cdot c^\pm(g)$. Let
\begin{align*}
\Ga^\pm:=\Gr\la S|c^\pm(R)\ra,
\end{align*}
where $c^\pm(R):=\{c^\pm(u)\cdot c^\pm(v)^{-1}:(u,v)\in R\}$. Of course $\Ga^\pm$ depends on $\la S|R\ra$, but we suppress this dependancy for lighter notations. Set
\begin{align*}
\ov c^\pm:L\to\Ga^\pm,~[t/s]\mapsto c^\pm(t)\cdot c^\pm(s)^{-1}.
\end{align*} 
Since $c^\pm$ is multiplicative, $\ov c^\pm$ doesn't change under growing $[t/s]$ into $[tf/sf]$. Since changing $[t/s]$ by a skein relation changes $\ov c^\pm([t/s])$ by an element of $c^\pm(R)$, it follows that $\ov c^\pm$ is well-defined group morphism. 
Finally, $\ov c^\pm$ is surjective as 
\begin{align*}
\ov c^-([\lambda_2(s)/\rho_2(s)])=ss\cdot s^{-1}=s\quad\textnormal{and}\quad \ov c^+([\rho_2(s)/\lambda_2(s)])=ss\cdot s^{-1}=s
\end{align*}
for any $s\in S$ and these elements generate $\Ga^\pm$. Hence, $\ov c^\pm:L\onto\Ga^\pm$ is a well-defined epimorphism. 

\begin{example}
Take $\cF$ with skein presentation being $S=\{a,b\}$ and $R$ the single relation 
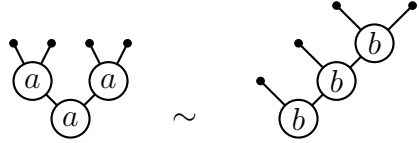
\begin{align*}
\SkeinRelationThree    
\end{align*}
We obtain $\Ga^-=\Gr\la a,b| aa=b\ra\simeq \Z$ and $\Ga^+=\Gr\la a,b|aa=bbb\ra\onto\Z_2*\Z_3$.
We will later take advantage of $\Ga^+$ having such a large quotient to obstruct piecewise projective actions. 
Moreover, the groups $\Ga^-,\Ga^+$ will induce invariants for distinguishing our groups.
\end{example}

\section{The examples}\label{sec:examples}
The remainder of the article is focused on proving the \ref{theo:main}. We first present the examples that will witness the result. Recall that $\lambda_n$ and $\rho_n$ are the left- and right-vines with $n$ carets, respectively, and for any monochromatic tree $t$ and colour $a$ that $t(a)$ denotes the tree $t$ having each interior vertex coloured by $a$. For each $n\geqslant3$ let
\begin{align}\label{eqn:tree}
\tau_n:=Y\circ(\lambda_{n-2}\ot Y)
\end{align}
and define the FS category $\cF_n:=\FS\la a,b|\tau_n(a)=\rho_n(b)\ra$. For example the skein relation defining $\cF_3$ is given by the diagram \ref{fig:skein-example} and for $\cF_4$ and $\cF_5$ the skein relations are 
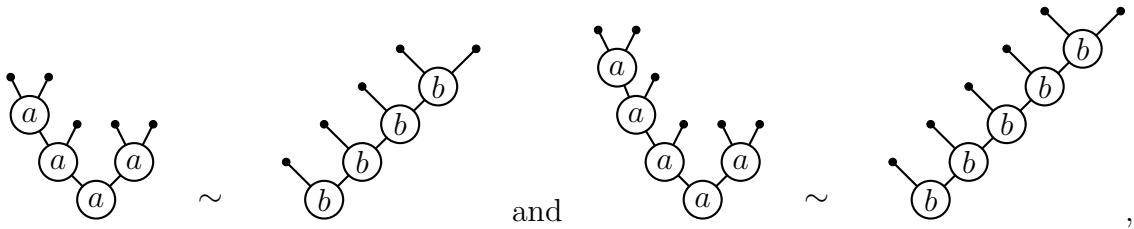
\begin{align*}
\SkeinRelationFour\quad\textnormal{and}\quad\SkeinRelationFive,
\end{align*}
respectively. By theorem \ref{theo:Ore-categories} each FS category $\cF_n$ is left-cancellative and has right-common-multiples. Hence, we may define the FS groups of $\cF_n$. 

\begin{notation}\label{notation}
We denote the $F,T,V$-type FS groups of $\cF_n$ by $L_n,G_n,M_n$, respectively, though our focus is on $G_n$. We may suppress the index $n$ if it is clear from context.
\end{notation}

\subsection{Description of the dynamics}\label{sec:dynamics}

In this section we study the dynamical aspects of the above groups. Recall $\fC:=\{\0,\1\}^\omega$ is the \emph{Cantor space}, which is totally ordered (lexiocographical order with $\0<\1$) and the order topology makes $\fC$ a compact, metrisable, totally disconnected space having no isolated points. It has a base of clopen intervals given by \emph{cones}: $w\cdot\fC=\{x\in\fC:w\cdot\ov\0\leqslant x\leqslant w\cdot\ov\1\}$, with $w\in\{\0,\1\}^*$. We will begin the section by giving a basic construction of the \emph{canonical action} $G_n\act\fC$ mentioned in subsection \ref{sec:simple} and for this we need certain \emph{infinite} partitions of $\fC$. Any FS group has a canonical action, though its construction is more involved in general, see \cite[Sec.~2]{Brothier-Seelig24a}. 
We then prove $G_n\act\fC$ is faithful. From the canonical action we then deduce a faithful action $G_n\act\bS$ by orientation-preserving homeomorphisms. We show $G_n\act\bS$ is the so-called \emph{McCleary--Rubin rigid action} for the group $G_n$ and calculate the groups of germs at every point in $\bS$. This calculation along with the McCleary--Rubin rigidity theorem allows us to distinguish every member of the family $(G_n)_{n\geqslant3}$.
We finish the section by describing the graph of one element of $G_3$.

\subsubsection{An infinite partition from an infinite tree}\label{sec:partition}

Fix $n\geqslant3$ and let $\fT_n$ be the infinite tree that is defined inductively by starting with the tree $\tau_n$ (defined in equation \ref{eqn:tree}) and gluing a copy of $\tau_n$ to the last leaf of the previous tree. For instance for $n=4$ we have
\begin{align}\label{fig:infinite-tree}
\tau_4=\Cell\quad\textnormal{and so}\quad \fT_4=\InfiniteQuasiRightVine.
\end{align}
The leaves of $\fT_n$ are totally ordered and are in bijection with $\N$. Recall that $\tau_n$ has $n+1$ leaves. If $\ell_j$ denotes the address of the $j$th leaf of $\tau_n$, then the address of the $i$th leaf of $\fT_n$ is given by $\mu_i=\1^{2m}\cdot\ell_j$,
where $i=mn+j$, $m\geqslant0$, $1\leqslant j\leqslant n$, and where $\1^{2m}$ means the word of length $2m$ with only $\1$'s. We don't express the dependence on $n$ in $\ell_j$ or $\mu_j$ as it will always be clear from context. It follows by definition that
\begin{align*}
\mu_1:=\ell_1=\0^{n-1}\quad\textnormal{and}\quad
\mu_{i+n}=\1\1\cdot\mu_i~\textnormal{for all}~i\geqslant 1.    
\end{align*}
The cones $(\mu_i\cdot\fC)_{i\geqslant1}$ furnish an infinite partition of $\fC\setminus\{\ov\1\}$. Precisely, $\cup_{i\geqslant 1} \mu_i\cdot \fC=\fC\setminus\{\ov\1\}$, where $\ov\1=\1\cdot\1\cdot\1\cdots$ is the maximum of $\fC$, and the family $(\mu_i\cdot\fC)_{i\geqslant1}$ is pairwise disjoint.
\subsubsection{Local actions}\label{sec:local-action} We now define \emph{local actions} for trees in $\cF_n$ using $\fT_n$ and its associated partition of $\fC\setminus\{\ov\1\}$. We start by defining some continuous order-preserving maps from $\fC$ to $\fC$ that are injective but not surjective. For all $x\in\fC$ let
\begin{align*}
A_\0(x):=\0\cdot x,\quad A_\1(x):=\1\cdot x,\quad B_\0(x):=\0^{n-1}\cdot x,
\end{align*}
and for all $i\geqslant1$ we define
\begin{align*}
B_\1(\mu_i\cdot x):=\mu_{i+1}\cdot x,\quad B_\1(\ov\1):=\ov\1,
\end{align*}
where recall that $\mu_i$ is the address of the $i$th leaf of $\fT_n$. Hence, the three first maps are rather basic while the $B_\1$ is more complex and acts by ``jumping'' from the $i$th cone to the $(i+1)$th cone in the partition given by the leaves of $\fT_n$.

{\bf Local action for free trees.} Let $t$ be a free tree and $i$ one of its leaves. Let $\eta_i$ be the unique path in $t$ from its root to the $i$th leaf, so $\eta_i=(\nu_1,\nu_2, \dots,\nu_{l+1})$, where $\nu_1=\varnothing$, $\nu_j$ is an interior vertex of $t$ for $1\leqslant j\leqslant l$, and $\nu_{l+1}$ is the address of the $i$th leaf of $t$. For each $1\leqslant j\leqslant l$ we have that $\nu_{j+1}=\nu_{j}\cdot \varepsilon_j$ for some $\varepsilon_j\in\{\0,\1\}$, i.e.,~the path $\eta_i$ either takes a left or right turn at each vertex. For each $1\leqslant j\leqslant l$ write $c_j\in\{a,b\}$ for the colour of the interior vertex $\nu_j$ in $t$.
To this vertex we associate the map $C^j_{\varepsilon_j}\in\{A_\0,A_\1,B_\0,B_\1\}$, where $C^j$ stands for the letter $A$ (resp.~$B$) if $c_j$ is $a$ (resp.~$b$). 
The \emph{local action of $t$ at $i$} is then
\begin{align}\label{eq:local-action}
\beta(t,i):=C^1_{\varepsilon_1}\circ C^2_{\varepsilon_2}\circ\cdots\circ C^l_{\varepsilon_l}:\fC\to\fC.
\end{align}
We freely identify such maps as words over $\{A_\0,A_\1,B_\0,B_\1\}$. For example, if
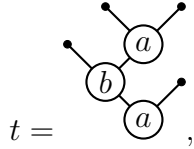
\begin{align*}
t=\LocalAction,
\end{align*}
then $\beta(t,1)=A_\0B_\0$, $\beta(t,2)=A_\0B_\1A_\0$, $\beta(t,3)=A_{\0}B_{\1}A_{\1}$, and $\beta(t,4)=A_\1$.

{\bf Local action for trees.} 
Recall $\lambda_n$ and $\rho_n$ are the monochromatic left and right vines with $n$ interior vertices, respectively, and that $\tau_n=Y\circ(\lambda_{n-2}\otimes Y).$ Moreover, recall we write $\tau_n(a)$ for the tree $\tau_n$ whose every interior vertex has been coloured by $a$. Equation \ref{eq:local-action} applied to $\tau_n(a)$ gives
\begin{align*}
\beta(\tau_n(a),1)=A_\0^{n-1},~\beta(\tau_n(a),i)=A_\0^{n-i}A_\1,~\beta(\tau_n(a),n)=A_\1A_\0,~\beta(\tau_n(a),n+1)=A_\1^2
\end{align*}
for all $1<i<n$ and for $\rho_n(b)$ we have
\begin{align*}
\beta(\rho_n(b),1)=B_\0,~\beta(\rho_n(b),j)=B_\1^{j-1}B_\0,~\beta(\rho_n(b),n+1)=B_\1^n
\end{align*}
for all $1<j\leqslant n$. The way we have defined the maps $A_\0,A_\1,B_\0,B_\1$ implies we have equality of the local actions $\beta(\tau_n(a),i)=\beta(\rho_n(b),i)$ for all $1\leqslant i\leqslant n+1$. We show this at the last leaf $i=n+1$. 
A similar argument applies to the other leaves. Since $\beta(\tau_n(a),n+1)=A_\1^2$ and $\beta(\rho_n(b),n+1)=B_\1^{n}$ for each $i\geqslant1$ and $x\in\fC$ we have that 
\begin{align*}
B_\1^{n}(\mu_i\cdot x)=\mu_{i+n}\cdot x
=\1\1\cdot \mu_i\cdot x
=A_\1^2(\mu_i\cdot x).
\end{align*}
The equalities follow similarly for the local actions at the rest of the leaves. This implies that local actions are well-defined for trees and not just for \emph{free} trees. Indeed, if $t$ is a free tree and $i$ is one of its leaves such that the path $\eta$ in $t$ from the root $\varnothing$ to the $i$th leaf (whose address is $\nu_i$) intersects a $\tau_n(a)$ or $\rho_n(b)$ subtree, then $\eta$ must pass through a unique leaf of $\tau_n(a)$ or $\rho_n(b)$, respectively. Suppose this is the $j$th leaf. Applying the skein relation to this $\tau_n(a)$ or $\rho_n(b)$ subtree changes a $\beta(\tau_n(a),j)$ or $\beta(\rho_n(b),j)$ subword in the local action $\beta(t,i)$ to $\beta(\rho_n(b),j)$ or $\beta(\tau_n(a),j)$, respectively. As these maps are equal local actions do not change under skein relations applied to $t$.

\subsubsection{Canonical action of an FS group}\label{sec:canonical-action} In this section we construct the canonical action for the $V$-type FS group $M$ of $\cF$ which restricts into actions on the other types. Firstly, for any tree $t$ with $l$ leaves the set of images of local actions $\{\beta(t,i)(\fC):1\leqslant i\leqslant l\}$ partitions $\fC$. This can be established via an induction on the number of carets of $t$ using the definitions of local actions. Now, given $g=[t/\pi/s]\in M$ we define $\al(g):\fC\to\fC$ on each image $\beta(s,i)(\fC)$ by the formula
\begin{align*}
\beta(s,i)(x)\mapsto\beta(t,\pi(i))(x),~x\in\fC.
\end{align*}
Since $A_\0,A_\1,B_\0,B_\1$ are homeomorphisms onto their images and each tree $t$ defines a partition of $\fC$, we deduce that $\al(g)$ is a well-defined homeomorphism and by construction this defines a group action $\al:M\act\fC$. 
We call it the \emph{canonical group action of $M$} and the restrictions of $\al$ to $G$ and $L$ are the \emph{canonical group actions of $G$ and $L$}, respectively. 

{\bf Diagrammatic calculus.} We can think of the canonical action as acting by ``local prefix replacement'' just like the standard action of Thompson's group $V$ on $\fC$, though our set of prefixes is richer. In terms of diagrams we think of $\beta(s,i)(x)$ as gluing the sequence $x\in\fC$ to the $i$th leaf of the tree $s$. The action of $[t/\pi/s]\in M$ on $\beta(s,i)(x)$ then amounts to gluing the diagram for $\beta(s,i)(x)$ on top of the diagram for $[t/\pi/s]=t\circ\pi\circ s^{-1}$. Cancelling the $s^{-1}\circ s$ we are left with $t\circ\pi$ with $x$ glued to the top of the $i$th strand of $\pi$. We then pull $x$ down to the leaves of $t$ via the $i$th strand of $\pi$. This results in $x$ glued to the $\pi(i)$th leaf of $t$, which is equal to the diagram for $\beta(t,\pi(i))(x)$. 

\begin{example}
For example, let $x\in\fC$, take the trees $s=Y_a$ and $t=Y_b$ (the unique carets coloured by $a$ and $b$, respectively), the first leaf $i=1$, the permutation $\pi=(12)$, and consider the group element $g=[Y_b/\pi/Y_a]$.
We obtain the following diagrammatic computation:
\begin{align*}
\al(g)(\beta(Y_a,1)(x))\quad=\quad\ActionOne\quad=\quad\ActionTwo\quad=\quad\ActionThree\quad=\quad\beta(Y_b,2)(x).
\end{align*}   
\end{example}

{\bf Dynamical characterisation of types of FS groups.} Consider the canonical action $\al:M\act\fC$ as constructed above. We have that $L$ (resp.~$G$) is the subgroup of $g\in M$ so that $\al(g)$ preserves the total order on $\fC$ (resp.~up to cyclic permutations), see \cite[Prop.~6.1]{Brothier22}. Moreover, $L$ is the subgroup of $g\in G$ 
such that $\al(g)$ fixes the endpoint $\ov\0\in\fC$, see the proof of item (3) of \cite[Prop.~6.2]{Brothier22}. 

{\bf Equivariant embedding of Thompson's groups.} The subgroup of all $a$-coloured diagrams inside $M$ is isomorphic to Thompson's group $V$, see \cite[Cor.~3.9]{Brothier22} for a proof in the $F$-case that easily adapts to the $V$-case. By the definition of $\al:M\act\fC$ we deduce that the restriction to this copy of $V$ is conjugate to the usual action of $V$ on Cantor space. We henceforth identify Thompson's groups $F,T,V$ inside $L,G,M$ by colouring interior vertices by $a$. In particular, we have the important observation:
\begin{observation}\label{obs:faithful}
Since the standard action of Thompson's groups on Cantor space is faithful, the canonical action $\al:M\act\fC$ restricted to $F\subseteq M$ is faithful.
\end{observation}

{\bf Dyadic rationals and rationals in Cantor space.} 
Let 
\begin{align*}
\fQ^-:=\{w\cdot\ov\0:w\in\{\0,\1\}^*\}\quad\textnormal{and}\quad\fQ^+:=\{w\cdot\ov\1:w\in\{\0,\1\}^*\}   
\end{align*}
which we call the set of \emph{left-} and \emph{right-dyadic rationals}, respectively. 
These correspond to the set of all left- and right-endpoints of all cones in $\fC$, respectively. It is well-known that these are the orbits of $\ov\0$ and $\ov\1$ under the usual action of Thompson's group $T$ as well as $V$. 
Now, by the definition of the canonical action, $M$ acts by locally order-preserving homeomorphisms of the Cantor space $\fC$. 
Hence, any element of $M$ will send left-endpoints to left-endpoints of cones (and similarly to right-endpoints of cones).
We deduce that $\fQ^\pm$ are globally preserved by $M$ and thus by $G$.
Since $T$ is already well-known to act transitively on $\fQ^\pm$ we deduce that $G$ acts transitively on $\fQ^\pm.$
We write the induced actions $\al^-:M\act\fQ^-$ and $\al^+:M\act\fQ^+$. We say $x\in\fC$ is \emph{rational} if $x=w\cdot\ov u$ for some words $w,u\in\{\0,\1\}^*$ with $u$ non-empty. Let $\fQ$ be the set of all rationals in $\fC$. In particular 
$\fQ^\pm\subset\fQ\subset\fC.$ 
Elements of $\fC\setminus\fQ$ are called \emph{irrational}.

\subsubsection{Faithfulness of canonical action}\label{sec:faithfulness}

We now prove the canonical group action $\al_n:M_n\act\fC$ is faithful for all $n\geqslant3$. 
Note that in general the canonical action of an FS group may \emph{not} be faithful --- a consequence of the diagrammatic calculus, see \cite[Ex.~4.4]{Brothier-Seelig24a}. 
The first part of the proof involves representing each $g\in M_n$ by a ``good'' tree-pair diagram. This was done previously in \cite[Sec.~6]{Brothier-Seelig24a} but we give a detailed proof here for convenience. The second part involves an analysis of the local actions.

{\bf Right-vine decomposition.} Any free tree $t$ can be uniquely decomposed as
\begin{align}\label{eq:decomposition}
t=t_1\circ\cdots\circ t_l,
\end{align}
where each forest $t_p$ is of the form $I^{\ot i_p}\ot\rho_{j_p}\ot I^{\ot k_p}$, where $\rho_{j_p}$ is a right-vine of length $j_p$ with some colouring, and $i_p<i_q$ if $1<q<p$, i.e.,~carets further left are drawn higher up. Note that $t_1=\rho_{j_1}$ as $t$ is a tree. This is a normal form for free trees and it comes from a complete rewriting system similar to the one for Thompson's monoid $F^+$, see for instance \cite[Sec.~2.2]{Dehornoy-Tesson19}. We call \ref{eq:decomposition} the \emph{right-vine decomposition of $t$}, $\rho_{j_p}$ the \emph{$p$th right-vine of $t$}, and the number of forests $l$ appearing in \ref{eq:decomposition} the \emph{right-vine number of $t$}.

{\bf Good trees.} Recall that $c^+(t)$ is the word over $\{a,b\}$ formed by reading the colouring along the unique path in $t$ from its root to its last leaf. A word over $\{a,b\}$ is \emph{good} if it is of the form $a^i\cdot w$, where $i\geqslant0$, and $w$ is either empty or a word starting with $b$ which does not contain $a^2$ or $b^n$ as a subword.
A free right-vine $\rho$ is \emph{good} if $c^+(\rho)$ is good and a free tree $t$ is \emph{good} if all of its right vines are good. In particular, any $a$-tree is good. In what follows we will say two free forests $f$ and $g$ are \emph{equivalent} if they are equal in $\cF$.

\begin{lemma}\label{lem:right-vine}
Up to growing by an $a$-forest (i.e.,~all interior vertices coloured $a$) any free right-vine is equivalent to a good tree all of whose vines have $a$-coloured roots. 
\end{lemma}
\begin{proof}
For lighter notations in this proof we denote $c:=c^+$, $\lambda:=\lambda_{n-2}(a)$, $\tau:=\tau_n(a)$, and $\rho:=\rho_n(b)$. Note that if $t_0$ is a free tree which represents a tree $t$ of $\cF$ and that we change a subtree $\tau$ of $t_0$ into $\rho$ (or the other way around), then we obtain a new free tree $t_1$ which still represents the same tree $t$ of $\cF$. This comes from having the skein relation $(\tau,\rho)$ in $\cF$.
We will use this fact many times throughout. Let $r$ be a free right-vine. We are going to inductively define a finite family of free trees so that each tree is equivalent up to growing by an $a$-forest to the previous one. Before doing this, we need to prepare $r$. For each $a$-vertex $\nu$ of $r$ attach to its left-child $\nu\cdot\0$ (which is a leaf of $r$) the free left-$a$-vine $\lambda$. Call the resulting free tree $r'$. Hence, whenever $\nu$ and $\nu\cdot\1$ are $a$-vertices in $r'$, we have a $\tau$ subtree of $r'$ rooted at $\nu$. Starting at the root of $r'$ and working towards its last leaf, we inductively change every $\tau$ subtree we come across into a $\rho$ subtree. This process terminates in a free tree $r''$ equivalent to $r'$ whose colouring along its right side is
\begin{align*}
c(r'')=b^{k_1}\cdot a\cdot b^{k_2}\cdot a\cdots b^{k_{p-1}}\cdot a\cdot b^{k_{p}},
\end{align*}
where $p\geqslant1$ and each $k_i$ is strictly positive except we may have $k_1=0$ or $k_{p}=0$. Up to growing the last leaf of $r''$ by the $a$-tree $\tau$ and then changing this into $\rho$, we may assume $k_p\geqslant n$. Now divide each $k_i$ by $n$ so that $k_i=nl_i+\epsilon_i$ for some $l_i\geqslant0$ and $0\leqslant\epsilon_i<n$. Since $k_p\geqslant n$ we have $l_p\geqslant1$. Let $w_{1}:=b^{k_1}\cdot a\cdots a\cdot b^{k_{p-1}}$ and consider the subword $a\cdot b^{k_p}=a\cdot b^{nl_p}\cdot b^{\epsilon_p}$ at the end of $w$, i.e.,~$c(r'')=w_{1}\cdot a\cdot b^{k_p}$. Let $\nu$ be the interior vertex of $r''$ associated to the $a$ appearing in this subword. Since $l_p\geqslant1$ there is a $\rho$ subtree of $r''$ rooted at $\nu\cdot\1$. Let $r'''$ be the result of changing this $\rho$ subtree of $r''$ into $\tau$. We have
\begin{align*}
c(r''')=w_{1}\cdot a\cdot a^2\cdot b^{n(l_p-1)}\cdot b^{\epsilon_p}=w_{1}\cdot a^2\cdot a\cdot b^{n(l_p-1)}\cdot b^{\epsilon_p}.   
\end{align*}
We have now ``prepared'' our right-vine $r$. Note that the vertex $\nu\cdot\0$ is now a leaf of $r'''$ (it was not a leaf of $r''$). If we grow this leaf by $\lambda$ we create a $\tau$ subtree rooted at $\nu$. Let $r_{1,1}$ be the result of growing $\nu\cdot\0$ in $r'''$ by $\lambda$ and then changing the subsequent $\tau$ rooted at $\nu$ into $\rho$. It follows that
\begin{align*}
c(r_{1,1})=w_1\cdot b^n\cdot a\cdot b^{n(l_p-1)}\cdot b^{\epsilon_p}.
\end{align*}
Let $w_2:=b^{k_1}\cdot a\cdots a\cdot b^{k_{p-2}}$ and continue inductively this process of changing $\rho$ subtrees and growing by $\lambda$'s. This eventually yields a free tree $r_{1,l_p}$ satisfying
\begin{align*}
c(r_{1,l_p})=w_2\cdot a\cdot b^{k_{p-1}+nl_p}\cdot a\cdot b^{\epsilon_p}=w_2\cdot a\cdot b^{n(l_{p-1}+l_p)+\epsilon_{p-1}}\cdot a\cdot b^{\epsilon_p}.
\end{align*}
Note that $\epsilon_p$ may be $0$, which is okay at this stage. We now isolate the part of the tree $r_{1,l_p}$ associated with the subword $a\cdot b^{k_{p-1}+nl_p}$ of $c(r_{1,l_p})$ and perform inductively growing and changing operations similar to that above. Eventually, this yields a free tree $r_{2,l_{p-1}+l_p}$ satisfying
\begin{align*}
c(r_{2,l_{p-1}+l_p})=w_2\cdot b^{n(l_{p-1}+l_p)}\cdot a\cdot b^{\epsilon_{p-1}}\cdot a\cdot b^{\epsilon_p}.
\end{align*}
Now if $\epsilon_{p-1}>0$, then we continue as before but now on the subword $a\cdot b^{n(l_{p-2}+l_{p-1}+l_p)}\cdot b^{\epsilon_{p-2}}$. However, if $\epsilon_{p-1}=0$, then there is an intermediate step we need to make. In this case 
\begin{align*}
c(r_{2,l_{p-1}+l_p})=w_2\cdot b^{n(l_{p-1}+l_p)}\cdot a^2\cdot b^{\epsilon_p}
\end{align*}
and we have now have a $\tau$ subtree in $r_{2,l_{p-1}+l_p}$ whose right-side corresponds to the $a^2$ above. Changing this subtree into $\rho$ gives a free tree whose right-side colouring is
\begin{align*}
w_2\cdot b^{n(l_{p-1}+l_p)}\cdot b^n\cdot b^{\epsilon_p}=w_2\cdot b^{n(l_{p-1}+l_p+1)}\cdot b^{\epsilon_p}.
\end{align*}
This concludes the intermediate step, and we continue as before, but now on this free tree. Continuing inductively, we eventually arrive at a free tree $r_{p}$ satisfying
\begin{align*}
c(r_{p})=b^{nm}\cdot b^{\ti\epsilon_1}\cdot a\cdot b^{\ti\epsilon_2}\cdot a\cdots a\cdot b^{\ti\epsilon_q},
\end{align*}
where $m=\sum_{i=1}^p l_i+1-\epsilon_i\geqslant1$ since $l_p\geqslant1$. Now $r_{p}$ has a rooted subtree given by $\rho_{nm}(b)$ which can be changed into an $a$-tree. Call the resulting free tree $t$. It follows by construction that $t$ has root colour is $a$, satisfies
\begin{align*}
c(t)=a^{2m}\cdot b^{\ti\epsilon_1}\cdot a\cdot b^{\ti\epsilon_2}\cdot a\cdots a\cdot b^{\ti\epsilon_q},
\end{align*}
with each $0<\ti\epsilon_i<n$, except we may have $\ti\epsilon_q=0$, $q\leqslant p$, and all the right-vines of $t$ (in its right-vine decomposition) except for its first right-vine are $a$-coloured, in particular, all right-vines of $t$ have $a$-coloured roots. Hence, $t$ is of the form we desire and by construction, it is equivalent to our initial right-vine $r$ up to growing by an $a$-forest.
\end{proof}

\begin{proposition}\label{prop:semi-normal}
Up to growing by an $a$-forest any free tree is equivalent to a good tree all of whose right-vines have their root coloured $a$.
\end{proposition}
\begin{proof}
We prove this by induction on right-vine number. We have shown the assertion is true for trees of right-vine number $1$ (i.e.,~right-vines) in lemma \ref{lem:right-vine}. Assume the assertion is true for all trees having right-vine number no greater than $k$ and let $t$ have right-vine number $k+1$. Let $t=t_1\circ t_2\circ\cdots\circ t_l$ be the right-vine decomposition of $t$ and denote $f:=t_2\circ\cdots\circ t_l$. It follows that all trees in the forest $f$ have no more than $k$ right-vines. Hence, by the inductive hypothesis, we may pick a $a$-forest $f'$ such that $ff'$ is equivalent to a forest $g$ all of whose trees are good and whose right-vines all have root colour $a$. 
We now want to grow $t_1\circ g$ by $a$-forests so as to change $c^+(t)$ into a good word. We would like do this in the same way as in the proof of lemma \ref{lem:right-vine}, however, we don't have the freedom to grow \emph{any} leaves of the right-vine $t_1$ owing to the possibility of a non-trivial tree from $g$ already being attached to said leaf. However, since we have arranged for all the right-vines of all the trees of $g$ to have root colour $a$, up to growing by more $a$-carets, we can guarantee to have $\tau_n(a)$ subtrees inside $t_1\circ g$ where we need them as in the proof of lemma \ref{lem:right-vine}. Hence, we are able to grow $t$ by an $a$-tree into a good tree all of whose vines have root colour $a$, as required.
\end{proof}

{\bf Seminormal form.} A free diagram $(t/s)$ is a \emph{seminormal form} if $s$ is an $a$-tree and $t$ is a good tree. A consequence of proposition \ref{prop:semi-normal} is that every element $g\in L$ is represented by a seminormal form. 
Indeed, if $(t/s)$ represents $g$, then by theorem \ref{theo:Ore-categories} we may pick some forest $f\in\cF$ such that $sf$ is equivalent to an $a$-tree. Now, by proposition \ref{prop:semi-normal} we may pick an $a$-forest $f'$ such that $tff'$ is a good tree. Since $sff'$ is still equivalent to an $a$-tree, we have that $(tff'/sff')$ is a seminormal form that represents $g$.

\begin{theorem}\label{theo:faithful}
The action $\al:M\act\fC$ is faithful. 
\end{theorem}

\begin{proof}
Recall $L,G,M$ are the $F,T,V$-type FS groups of $\cF$. First we observe that $\ker(\al)\subset L$ as any element of $\ker(\al)$ preserves the total order on $\fC$, which characterises $L$ inside $M$, see \cite[Prop.~6.1]{Brothier22}. Let $\cB$ be the set of all seminormal forms representing elements of $\ker(\al)$ that contain at least one $b$-vertex. Suppose for a contradiction that $\cB$ is non-empty and let $(t/s)\in\cB$ have minimal number of carets among elements of $\cB$. Let $t=t_1\circ\cdots\circ t_l$ be the right-vine decomposition for $t$. Let us first make the observation that the local action of $t$ at its last leaf is the same as the local action of $t_1$ at its last leaf. Now, by definition $c^+(t_1)=a^i\cdot w$, where $w$ is either empty or starts with $b$ and contains no $a^2$ or $b^n$ subword. If $i>0$, then we can construct an element of $\cB$ with fewer carets. Indeed, if $i>0$, then we may write $t=Y_a(t'\ot t'')$ and $s=Y_a(s'\ot s'')$. Since $(t/s)$ represents a kernel element, we must have that $t'$ and $s'$ as well as $t''$ and $s''$ have the same number of leaves. Hence $(t'/s')$ and $(t''/s'')$ are well-defined diagrams and they must represent elements of the kernel since $(t/s)$ does. Moreover, both diagrams have fewer carets than $(t/s)$, at least one of those diagrams contains a $b$-vertex, and both are in seminormal form. This contradicts minimality of $(t/s)$, so we must have $i=0$. In particular, $w$ cannot be empty (as then $t_1$ would have no interior vertices). Now $w$ is a word that starts with $b$ and contains no $a^2$ or $b^n$ subword and $c^+(t_1)=w$. Either $w$ does not contain an $a$ in which case $w=b^{i}$ for some $0<i<n$ or it does contain an $a$ and it is of the form $w=b^{i_1}(ab^{i_2})(ab^{i_3})\cdots(ab^{i_k})a^{\varepsilon}$, where each $0<i_j<n$, and $\varepsilon\in\{0,1\}$. Let $\ell$ be the number of leaves of $t_1$. We claim the local action of $t_1$ at $\ell$ never acts like the local action of an $a$-tree at its last leaf, i.e.,~like $A_\1^k$ for some $k>0$. To see this we will use the definition of the local actions outlined in subsection \ref{sec:local-action}. By definition, the local action $\beta(t_1,\ell)$ is the word $W$ over $\{A_\1,B_\1\}$ formed by substituting $A_\1$ for $a$ and $B_\1$ for $b$ into the word $w$. In the first case $W=B_\1^{i}$ for $0<i<n$ in which case
\begin{align*}
B_\1^{i}(\1\0\cdot x)=
\begin{cases}
\1\1\0^{n-1}\cdot x&\textnormal{if $i=1$,}\\
\1\1\0^{n-i}\1\cdot x&\textnormal{otherwise,}
\end{cases}
\end{align*}
for any $x\in\fC$. Since $n\geqslant3$ and $0<i<n$ we have that $B_\1^i(\1\0\cdot x)$ always contains a $\0\0$ whereas $A_\1^k(\1\0\cdot x)$ may not. Hence, $B_\1^i\not=A_\1^k$ for any $k>0$. The other case is
\begin{align*}
W=B_\1^{i_1}(A_\1B_\1^{i_2})(A_\1B_\1^{i_3})\cdots(A_\1B_\1^{i_k})A_\1^{\varepsilon},    
\end{align*}
where each $0<i_j<n$, and $\varepsilon\in\{0,1\}$. Suppose $\varepsilon=0$ and write $\nu_p$ for the $p$th leaf of the left-vine of length $n-2$. Since $A_\1^2=B_\1^n$ commutes with $A_\1$ and $B_\1$ we have that $Z(\1\1\cdot x)=\1\1\cdot Z(x)$ for any $Z\in\{A_\1,B_\1\}^*$ and $x\in\fC$. Now, we deduce
\begin{align*}
W(\0\1\cdot x)&=B_{\1}^{i_1}(A_\1B_\1^{i_2})\cdots(A_\1B_\1^{i_k})(\1\1\1\0\cdot\nu_{i_k}\cdot x)\\
&=\1\1\cdot B_{\1}^{i_1}(A_\1B_\1^{i_2})\cdots(A_\1B_\1^{i_{k-1}})(\1\0\cdot\nu_{i_{k}}\cdot x)\\
&\hspace{2.25mm}\vdots\\
&=\1^{2k}\0\cdot\nu_{i_1}\nu_{i_2}\cdots\nu_{i_k}\cdot x,
\end{align*}
for all $x\in\fC$. Similar to as above, $x$ can be chosen so that the sequence $W(\0\1\cdot x)$ can have more $\0$'s than the sequence $A_\1^k(\0\1\cdot x)$ for any $k>0$. Hence, in this case we have $W\not=A_\1^k$ for any $k>0$. If $\varepsilon=1$, then we input a sequence of the form $\1\0\1\cdot x$ and are able to conclude similarly as the previous case. All in all we have shown that the local action of $t_1$ at its last leaf (which is equal to the local action of $t$ at its last leaf) is never equal that of an $a$-tree, hence $(t/s)$ cannot be in $\cB$, which is a contradiction.

As such, $\cB=\varnothing$. The seminormal forms that remain to represent elements of $\ker(\al)$ are those diagrams containing only $a$-vertices. However, by observation \ref{obs:faithful} no non-trivial element made from $a$-carets can be in $\ker(\al)$. Hence, the only diagrams representing elements of $\ker(\al)$ are of the form $(t/t)$ and all such represent the identity element. This shows $\ker(\al)=\{e\}$, so the canonical action $\al_n:M_n\act\fC$ is faithful for all $n\geqslant3$.
\end{proof}

\subsubsection{McCleary--Rubin rigid actions}\label{sec:rubin-FS} McCleary and Rubin outline mild conditions for a group to admit a canonical action which can be constructed only from the algebraic data of the group itself, see \cite{McCleary78,Rubin89,McCleary-Rubin05,Belk-Elliott-Matucci25}. Such a group then enjoys powerful dynamical invariants such as groups of germs and orbits. These results have been applied to distinguish many Thompson-like groups, see for instance \cite{Bleak-Lanoue10}. We use a variant of the Rubin--McCleary theorem adapted actions of groups on the circle by orientation-preserving homeomorphisms (henceforth OPH action).

A group action on the circle $\Ga\act\bS$ is \emph{McCleary--Rubin rigid} or \emph{rigid} for short if it is faithful, OPH, and satisfies the following two properties:
\begin{itemize}
    \item There exists a dense subset $D\subset\bS$ such that for any positively oriented triples $(x,y,z),(x',y',z')\in D^3$ there exists $g\in\Ga$ such that $(g(x),g(y),g(z))=(x',y',z')$.
    \item There exists a non-trivial $g\in\Ga$ that fixes a non-empty open set.
\end{itemize}

The first property is called ``being 3-o-transitive on $D$'' and the second ``having an element of bounded support''. One can deduce the following from \cite[Thm.~7.20]{McCleary-Rubin05}.
\begin{theorem}\label{theo:C-rigid}
Let $\ga_1:\Ga_1\act\bS$ and $\ga_2:\Ga_2\act\bS$ be two rigid actions. If $\phi:\Ga_1\to\Ga_2$ is an isomorphism, then there exists a unique homeomorphism $\Phi:\bS\to\bS$ such that
\begin{align*}
\ga_2(\phi(g))=\Phi\circ\ga_1(g)\circ\Phi^{-1} \text{ for all } g\in\Ga_1.
\end{align*}
\end{theorem}

As a rigid action $\Ga\act X$ is unique up to conjugation, we may call $\Ga\act X$ ``the'' rigid action of $\Ga$.

{\bf The rigid action of a $T$-type FS group.} We consider the circle $\bS$ as $\R/\Z$. The map $\Sigma:\fC\to \bS$ defined on $\fC\setminus\{\ov\1\}$ by the formula
\begin{align}\label{eqn:Cantor-to-circle}
\Sigma(x_1x_2\cdots):=\sum_{k\geqslant1}\frac{x_k}{2^k},
\end{align}
and otherwise by $\Sigma(\ov\1)=0$, is continuous, cyclic order-preserving, and surjective. Setting $\Z[\frac{1}{2}]:=\{
m 2^{-k}:m,k\in\Z\}$ we call $\Q_2:=\Z[\frac{1}{2}]/\Z$ the set of all \emph{dyadic rationals in $\bS$}. The preimage of a non-zero $p\in\Q_2$ under $\Sigma$ is of the form
\begin{align*}
\{w\cdot\0\ov\1,w\cdot\1\ov\0\}
\end{align*}
for some $w\in\{\0,\1\}^*$. Otherwise $\Sigma^{-1}(0)=\{\ov\0,\ov\1\}$ and if $p\not\in\Q_2$ its preimage is a singleton. By definition $\fQ^-$ is the set of infinite binary strings eventually constant equal to $\0$.
Hence, $\Sigma(\fQ^-)$ is the set of all finite sums as above and thus is equal to all dyadic rationals of the circle.
Similarly, if $w\in \fQ^+$, we can decompose it as $w=v\cdot \overline \1$ and thus $$\Sigma(w)=\Sigma(v) + \frac{\Sigma(\ov\1)}{2^l}=\Sigma(v)+2^{-l}$$ where $l$ is the length of $v.$
Hence, we similarly deduce that $\Sigma(\fQ^+)=\Q_2.$
We obtain that
$$\Q_2=\Sigma(\fQ^-)=\Sigma(\fQ^+).$$ For each $g\in G$ we define a bijection $\ga(g):\bS\to\bS$ by
\begin{align*}
\ga(g)(\Sigma(x)):=\Sigma(\al(g)(x)),~x\in\fC.
\end{align*}
Since the permutations defining elements of $G$ are cyclic, it follows that $\ga(g)$ is a homeomorphism of $\bS$, and so this formula furnishes an OPH action $\ga:G\act\bS$. Since $\fQ^-$ is stable under $\al:G\act\fC$ the same is true for $\Q_2$ under $\ga:G\act\bS$ as $\Q_2=\Sigma(\fQ^-)$. Hence, there is an induced action $\ga:G\act\Q_2$ which restricts to the standard action of Thompson's group $T$ on dyadic rationals. Moreover, if $\al$ is faithful, then so is $\ga$.
Indeed, suppose for any $g\in G$ there exists $x\in\fC$ such that $\al(g)(x)\not=x$. Since $\fC$ is Hausdorff and $\al(g)$ is continuous we may pick an open interval $I$ containing $x$ such that $\al(g)(I)\cap I=\varnothing$. Now, $J:=\Sigma(I)$ is an open set in $\bS$ containing $\Sigma(x)$ and satisfies $\ga(g)(J)\cap J=\varnothing$. Hence, $\ga$ is faithful. Moreover, $L$ is the stabiliser of $0\in\bS$ inside $G$.

\begin{theorem}\label{theo:rigid}
The actions $\ga:G\act\bS$ and $\ga:D(G)\act\bS$ are McCleary--Rubin rigid.
\end{theorem}
\begin{proof}
That $\ga:G\act\bS$ is faithful OPH follows from theorem \ref{theo:faithful} and the above discussion. Hence, the restricted action $\ga:D(G)\act\bS$ is also faithful OPH. Colouring monochromatic tree-diagrams by the colour $a$ furnishes an embedding $T\into G$ that respects the actions on $\Q_2\subset\bS$. Since $T$ is perfect (as it is simple) we deduce an equivariant embedding $T\into D(G)$. It is known that $T\act\Q_2$ is $3$-o-transitive and has elements of bounded support, see \cite[Lem.~1.2.3]{Brin-Guzman98}. Hence, these properties are also satisfied by $D(G)\act\Q_2$ and $G\act\Q_2$. Since $\Q_2\subset\bS$ is dense $\ga:D(G)\act\bS$ and $\ga:G\act\bS$ are rigid.
\end{proof}

\subsubsection{Groups of germs} 
Given a faithful group action $\Ga\act X$ and $x\in X$ the \emph{germ of $\ga$ at $x$} is the set $[\ga]_x$ of all $\eta\in\Ga$ that act the same as $\ga$ on a neighbourhood of $x$. We set
\begin{align*}
[\Ga;X]_x:=\{[\ga]_x:\ga\in\Ga,\ga(x)=x\},
\end{align*}
which is a group under $[\ga]_x[\eta]_x:=[\ga\eta]_x$ called the \emph{group of germs of $\Ga\act X$ at $x\in X$}. We may write $[\Ga]_x$ instead of $[\Ga;X]_x$ if the action is clear from context. 

{\bf Basic properties of groups of germs.} The group of germs at a point only depends on the orbit of that point. Indeed, if $\ga\in\Ga$ satisfies $\ga(x)=y$, then
\begin{align*}
[\Ga]_x\to[\Ga]_y,~[\eta]_x\mapsto[\ga\eta\ga^{-1}]_y
\end{align*}
is an isomorphism. For any $x\in X$, a subgroup $\Delta\subset\Ga$ induces a subgroup $[\Delta]_x\subset[\Ga]_x$.

{\bf Group of germs for Thompson's group $T$.} Recall the canonical action of Thompson's group $T$ on the Cantor space $\fC$ is by local finite prefix replacement which preserves the cyclic order. Given a word $w\in\{\0,\1\}^*$ with $m\geqslant1$ letters, an element $g\in T$ fixes $\ov w:=w\cdot w\cdot w\cdots$ ($w$ concatenated with itself infinitely many times) if and only if there exist $i,j\geqslant1$ such that $g(w^j\cdot x)=w^i\cdot x$ for all $x\in\fC$. Consequently, the map
\begin{align*}
[T]_{\ov w}\to m\Z,~[g]_{\ov w}\mapsto m(i-j)
\end{align*}
is an isomorphism. Since the group of germs of $x$ only depends on the orbit of $x$ we deduce that $[T]_{x}\simeq\Z$ for all rational points $x$ (including dyadics). Since elements of $T$ can only change finite prefixes of elements in $\fC$ it follows that $g\in T$ can only fix an irrational point if it fixes a neighbourhood of that point. Thus, $[T]_x\simeq\{e\}$ if $x$ is irrational. We refer the reader to \cite[Sec.~3]{Brin96} for details, particularly the discussion after lemma 3.3.

{\bf Group of germs for the canonical action of an FS group.} Recall that $L,G$ are the $F$-type and $T$-type FS groups of our fixed FS category $\cF$, respectively. Recall the quotients $\ov c^\pm:L\to\Ga^\pm$ from section \ref{sec:quotients} and let $K^-,K^+\subset L$ consist of those elements fixing a neighbourhood of $\ov\0$ and $\ov\1$ under the canonical action $\al:L\act\fC$, respectively. 

\begin{lemma}\label{lem:germ}
Since the canonical action is faithful, we have $\ker(\ov c^\pm)=K^\pm$.    
\end{lemma}

\begin{proof}
We prove the assertion only in the ``$+$'' case as the ``$-$'' case follows from a similar argument. Let $[t/s]\in\ker(\ov c^+)$ and let $l$ be the number of leaves of $t$ and $s$. By definition of $\ov c^+$ we have that $c^+(t)=c^+(s)$ in the group $\Ga^+$. This implies there is a finite sequence of applications of the relations $a^2=b^n$, $xx^{-1}=e$, $x^{-1}x=e$, $x\in\{a,b\}$, taking the word $c^+(t)$ to the word $c^+(s)$. Since $A_\1^2=B_\1^n$, we deduce a finite sequence of relations taking the local action $\beta(t,l)$ to $\beta(s,l)$. It follows that $[t/s]$ fixes $\beta(s,l)(\fC)$, which contains a neighbourhood of $\ov\1$, so $[t/s]\in K^+$. This establishes the inclusion $\ker(\ov c^+)\subset K^+$, which we note holds \emph{regardless} of faithfulness of the canonical action.

Conversely, suppose $[t/s]\in K^+$ and let $l$ be the number of leaves of $t$ and $s$. By growing the representative $(t/s)$ we may assume that $[t/s]$ fixes the range of $\beta(s,l)$. Since $\cF$ has right-common-multiples, we pick $p,q\in\cF$ such that $tp=sq$. Now $[t/s]=[sq/sp]$ and we must have that the last trees of $p$ and $q$ have the same number of leaves as otherwise $[sq/sp]$ would not fix $\beta(s,l)(\fC)$. Hence, $[q_l/p_l]\in L$ is well-defined. Now we use the assumption that $\al:L\act\fC$ is faithful. Since $[sq/sp]$ fixes $\beta(s,l)(\fC)$ we have $[q_l/p_l]\in\ker(\al)$ so $[q_l/p_l]=e$. Now reading the colours off the right side of $sq$ and $sp$ yields the same word, so $[t/s]=[sq/sp]\in\ker(c^+)$, which gives $K^+\subset\ker(\ov c^+)$ and thus $\ker(\ov c^+)=K^+$. 
\end{proof}

Recall the right-dyadics $\fQ^+=\{w\cdot \ov\1: w\in\{\0,\1\}^*\}$ from Section \ref{sec:canonical-action} and the notation for the germ $[g]_x$ of a group element $g\in G$ at a point $x\in\fC$. Recall also $\fT=\fT_n$ from subsection \ref{sec:partition} is an infinite tree whose leaves $(\mu_j)_{j\geqslant1}$ give an infinite partition of $\fC\setminus\{\ov\1\}$.

\begin{lemma}\label{lem:a-germ}
For all $x\in\fC\setminus\fQ^+$ and $g\in G$ there exist $a$-trees $t(a),s(a)$, and a cyclic permutation $\pi$ of their leaves giving a group element $h:=[t(a)/\pi/s(a)]\in G$ satisfying $[g]_x=[h]_x$. Hence, at any non-right-dyadic point of $\fC$ any element of $G$ acts locally as an element of Thompson's group $T$.
\end{lemma}
\begin{proof}
Fix $x\in\fC\setminus\fQ^+$. 

{\bf Claim: For any tree $t$ and any leaf $i$ there exist words $u,v\in\{\0,\1\}^*$ such that $x\in u\cdot\fC$ and 
$\beta(t,i)(u\cdot y)=v\cdot y$ for all $y\in\fC$.}

We induct on the length of the local action $\beta(t,i)$ as a word in $\{A_\0,A_\1,B_\0,B_\1\}$. The base case is clear for $A_\0$, $A_\1$, and $B_\0$ (take $u=\varnothing$), we just need to check for $B_\1$. Since $x\in\fC\setminus\fQ^+$ it follows that $x$ ``passes through'' a leaf of the infinite tree $\fT$, that is, $x=\mu_j\cdot x'$ for some $j\geqslant1$ and $x'\in\fC$. Taking $u=\mu_j$ and $v=\mu_{j+1}$ gives the claim for $B_\1$. Indeed, $x\in\mu_j\cdot\fC$ and $B_\1(\mu_j\cdot y)=\mu_{j+1}\cdot y$ for any $y\in\fC$. Suppose the claim is true for those local actions having length strictly less than $m>1$ and let $\beta(t,i)$ have length $m$. We decompose $\beta(t,i)=Z\circ W$, where $Z\in\{A_\0,A_\1,B_\0,B_\1\}$ and $W$ has length $m-1$. By the inductive hypothesis we may choose $u,v\in\{\0,\1\}^*$ such that $x\in u\cdot\fC$ and $W(u\cdot y)=v\cdot y$ for all $y\in\fC$. If $Z\in\{A_\0,A_\1,B_\0\}$ then we are done. By assumption $x=u\cdot x'$ for some $x'\in\fC\setminus\fQ^+$. Since $x'$ is not eventually $\ov\1$ we may pick a prefix $v'$ of $x'$ such that $u\cdot v'=\mu_l\cdot v''$ for some $l\geqslant1$ and $v''\in\{\0,\1\}^*$.
Now $x\in uv'\cdot\fC$ and by definition of $W$
\begin{align*}
\beta(t,i)(uv'\cdot y)=B_\1(W(u\cdot v'\cdot y))=B_\1(v\cdot v'\cdot y)=B_\1(\mu_l\cdot v''\cdot y)=\mu_{l+1}v''\cdot y
\end{align*}
for any $y\in\fC$. This completes the proof of the claim. 

We can now prove the lemma. Fix $x\in\fC\setminus\fQ^+$ and $g\in G$. By theorem \ref{theo:Ore-categories} any tree can be grown into an $a$-tree and so we may assume that $g$ is represented by a diagram $(t/\pi/s)$ where $s$ is $a$-coloured. We may write $x=u\cdot x'$ where $u$ is the address of the $i$th leaf of $s$ (so that $\beta(s,i)(y)=u\cdot y$ for all $y\in\fC$) and $x'\in\fC\setminus\fQ^+$. By the claim we pick $u',v'\in\{\0,\1\}^*$ such that $x'\in u'\cdot\fC$ and $\beta(t,\pi(i))(u'\cdot y)=v'\cdot y$ for all $y\in\fC$. Now
\begin{align*}
\al(g)(uu'\cdot y)=\beta(t,\pi(i))\circ\beta(s,i)^{-1}(u\cdot u'\cdot y)=\beta(t,\pi(i))(u'\cdot y)=v'\cdot y
\end{align*}
for all $y\in\fC$. We now construct a $a$-coloured tree-diagram $h$ sending $uu'\cdot\fC$ affinely to $v'\cdot\fC$. Let $t$ be a tree containing $v'$ as a leaf a tree and $s$ another tree containing $uu'$ as a leaf and having the same number of leaves as $t$. Let $\pi$ be the cyclic permutation taking the leaf $v'$ in $t$ to $uu'$ in $s$. Now $h:=[t(a)/\pi/s(a)]$ sends $uu'\cdot\fC$ affinely to $v'\cdot\fC$. Since $x\in uu'\cdot\fC$ we have $[g]_x=[h]_x$.
\end{proof}

We now calculate all groups of germs for the canonical action of $G$ as well as of $D(G)$.

\begin{proposition}\label{prop:canonical-germ} For the canonical action $\al:G\act\fC$ the following are true.
\begin{enumerate}
    \item For all $x^\pm\in\fQ^\pm$ we have $[G;\fC]_{x^\pm}\simeq\Ga^\pm$. 
    \item For all $x\in\fQ\setminus(\fQ^-\cup\fQ^+)$ we have $[G;\fC]_x\simeq\Z$.
    \item For all $x\in\fC\setminus\fQ$ we have $[G;\fC]_x\simeq\{e\}$.
    \item For all $x\in\fC$ we have $[D(G);\fC]_x\simeq[G;\fC]_x$.
\end{enumerate}
\end{proposition}
\begin{proof}
(1) We show this only for $x\in\fQ^-$ as the other case follows similarly. Recall the $F$-type group $L$ is the subgroup of $G$ whose elements fix $\ov\0$. By the definition of groups of germs, lemma \ref{lem:germ}, and the quotient $c^-:L\to\Ga^-$, we deduce that
\begin{align*}
[G]_{\ov\0}=L/K^-=L/\ker(c^-)\simeq\Ga^-.
\end{align*}
Since any $x\in\fQ^-$ is in the same $G$-orbit as $\ov\0$ we have $[G]_x\simeq[G]_{\ov\0}\simeq\Ga^-$, as required.

(2) Let $x\in\fQ\setminus(\fQ^-\cup\fQ^+)$. The subgroup $G(a)\subset G$ of all $a$-coloured tree-diagrams induces an inclusion $[G(a)]_x\subset[G]_x$ and lemma \ref{lem:a-germ} implies $[G]_x\subset[G(a)]_x$, so $[G]_x=[G(a)]_x$. As isomorphism $T\simeq G(a)$ respects canonical actions and $[T]_x\simeq\Z$, we deduce $[G]_x\simeq\Z$.

(3) Let $x\in\fC\setminus\fQ$. As in item (2), lemma \ref{lem:a-germ} implies $[G]_x\simeq[T]_x$, and since $[T]_x\simeq\{e\}$ for irrational points, we deduce $[G]_x\simeq\{e\}$.

(4) For any $x\in\fC$ the subgroup $D(G)\subset G$ induces an inclusion $[D(G)]_x\subset[G]_x$. Given $g\in G$ one can construct $h\in G$ acting the same as $g$ on some neighbourhood $U$ of $x$ but that fixes some non-empty open set. Up to taking a proper subset of $U$, it follows that there exists a $k\in G$ such that $k(\supp(h)\cup U)\cap(\supp(h)\cup U)=\varnothing$, and so $[h,k]\in D(G)$ acts like $g$ on $U$. We refer the reader to Claim 1 of \cite[Thm.~3.5]{Brothier-Seelig24a} for further details of such a construction. This is to say that the germ of $g$ at $x$ can be represented by an element of $D(G)$, which completes the claim.
\end{proof}

{\bf Groups of germs for the rigid action.} We have computed groups of germs for the canonical action $G\act \fC.$ 
We will now deduce groups of germs for the rigid action $G\act\bS$.
Recall the continuous surjective map $\Sigma:\fC\onto\bS$ defined in subsection \ref{sec:rubin-FS}.

\begin{proposition}\label{prop:rigid-germ}
For the rigid action $\ga:G\act\bS$ the following are true.
\begin{enumerate}
    \item For all $x\in\Q_2$ we have $[G;\bS]_x\simeq\Ga^+\times\Ga^-$.
    \item For all $x\in\bS\setminus\Q_2$ we have $[G;\bS]_{x}\simeq[G;\fC]_{y}$, where $\Sigma^{-1}(x)=\{y\}$.
    \item For all $x\in\bS$ we have $[D(G);\bS]_x\simeq[G;\bS]_x$.
\end{enumerate}
\end{proposition}
\begin{proof}
(1) As the isomorphism type of the group of germs at $x$ only depends on its orbit, it suffices to show $[G;\bS]_0\simeq\Ga^+\times\Ga^-$. Recall $\Sigma^{-1}(0)=\{\ov\0,\ov\1\}$. Consider the map
\begin{align*}
\phi:[G;\bS]_0\to[G;\fC]_{\ov\1}\times[G;\fC]_{\ov\0},~[g]_0\mapsto([g]_{\ov\1},[g]_{\ov\0}).
\end{align*}
If $\ga(g)$ fixes a neighbourhood $U$ of $0$, then $\Sigma^{-1}(U)$ contains a neighbourhood of $\{\ov\0,\ov\1\}$ fixed by $\al(g)$. Hence, $\phi$ is well-defined, and is a group morphism as it factors through the diagonal morphism and canonical quotients. We claim $\phi$ is an isomorphism. For injectivity, suppose $\phi([g]_0)=([g]_{\ov\1},[g]_{\ov\0})=([e]_{\ov\1},[e]_{\ov\1})$, so that $\al(g)$ fixed a neighbourhood $C$ of $\{\ov\0,\ov\1\}$. It follows that $\Sigma(C)$ contains a neighbourhood of $0$ fixed by $\ga(g)$ and so $[g]_0=[e]_0$. Surjectivity is slightly more involved.

{\bf Claim: For any $g\in L$ there is $g^+\in L$ satisfying $[g^+]_{\ov\1}=[g]_{\ov\1}$ and $[g^+]_{\ov\0}=[e]_{\ov\0}$.}

Fix $g\in L$ and pick trees $t,s\in\cF$ with $l\geqslant3$ leaves such that $g=[t/s]$. Since $\cF$ has right-common-multiples, pick $p,q\in\cF$ be such that $tp=sq$. Let $l_p$ and $l_q$ be the number of leaves of the first tree of $p_1$ of $p$ and the first tree $q_1$ of $q$, respectively. Let $p'$ and $q'$ be any trees satisfying $|\Leaf(p')|+l_p=|\Leaf(q')|+l_q$ and define
\begin{align*}
p'':=p_1\ot p'\ot I^{\ot l-2}\quad\textnormal{and}\quad q'':=q_1\ot q'\ot I^{\ot l-2},    
\end{align*}
both of which have $l$ roots. Since $p''$ and $q''$ have the same number of leaves, it makes sense to form the fraction of trees $g^+:=[tp''/sq'']\in L$. By inspecting local actions at the first and last leaves, we deduce $[g^+]_{\ov\0}=[e]_{\ov\0}$ and $[g^+]_{\ov\1}=[g]_{\ov\1}$, as claimed.

By a similar argument, given $g\in L$ there is $g^-\in L$ such that $[g^-]_{\ov\1}=[e]_{\ov\1}$ and $[g^-]_{\ov\0}=[g]_{\ov\0}$. Fix $g,h\in L$ and let $g^+,h^-\in L$ be as defined above. We have
\begin{align*}
[g^+h^-]_{\ov\1}=[g^+]_{\ov\1}[h^-]_{\ov\1}=[g^+]_{\ov\1}=[g]_{\ov\1}
\end{align*}
and similarly $[g^+h^-]_{\ov\0}=[h]_{\ov\0}$. Hence, $[g^+h^-]_0$ maps to $([g]_{\ov\1},[h]_{\ov\0})$ under $\phi$, and so $\phi$ is an isomorphism, and we conclude using item (1) of proposition \ref{prop:canonical-germ}.

(2) If $x\in\bS\setminus\Q_2$, then $\Sigma^{-1}(x)=\{y\}$ for some $y\in\fC\setminus(\fQ^+\cup\fQ^-)$. Now 
\begin{align*}
\psi:[G;\bS]_{x}\to[G;\fC]_y,~[g]_{x}\mapsto[g]_y
\end{align*}
is a well-defined as group injective group morphism for similar reasons to above.
It is surjective since it factors through the identity and a quotient.

(3) The argument is similar to that of item (5) of proposition \ref{prop:canonical-germ}.
\end{proof}

\begin{remark}
Morally, item (1) of proposition \ref{prop:rigid-germ} says the local behavior of an element of $G$ at a ``breakpoint'' (endpoint of cone) is given by its local behaviour to the left and by its local behaviour to the right. Moreover, these left and right local behaviours may be different (like the slope to the left and the slope to the right of a breakpoint for Thompson's group $T$). 
This is in contrast to item (2) where the local behaviour at any other point is the same on the left and the right.
\end{remark}

\subsubsection{An example of the action of an element}\label{sec:graph}

Let $\bS=\R/\Z$ and identify the half open interval $[0,1)$ inside $\bS$ in the natural way. We will be considering closed intervals inside $[0,1)$ and we write them $[x,y]=\{z\in[0,1):x\leqslant z\leqslant y\}$. Consider the category $\cF_3$ as defined at the beginning of this section and let $g:=[Y_b/Y_a]\in G_3$. We have that $\ga(g):\bS\to\bS$ lifts to $\al(g):\fC\to\fC$. The action on the cone $\0\cdot\fC$ is straightforward
\begin{align*}
\al(g)(\0\cdot x)=B_\0A_\0^{-1}(\0\cdot x)
=B_\0(x)=A_\0^2(x)
=\0\0\cdot x.
\end{align*}
Hence, $\ga(g)$ maps $[\Sigma(\0\cdot\ov\0),\Sigma(\0\cdot\ov\1)]=[0,\frac{1}{2}]$ to $[\Sigma(\0\0\cdot\ov\0),\Sigma(\0\0\cdot\ov\1)]=[0,\frac{1}{4}]$ in the unique affine way: $z\mapsto z/2$.
Things are not so straightforward on $\1\cdot\fC$. Indeed
\begin{align*}
\al(g)(\1\cdot x)&=B_\1A_\1^{-1}(\1\cdot x)=B_\1(x)
\end{align*}
and we do not know how $B_\1$ acts on $x$ without further specification on $x$.  This was exactly the reason for introducing the infinite partition of $\fC$. In this case, the partition of $\fC$ is given by the leaves of the infinite tree appearing in diagram \ref{fig:infinite-tree}. In this case, the first three leaves of $\fT_3$ are given by $\mu_1=\0\0$, $\mu_2=\0\1$, and $\mu_3=\1\0$. Now, for any $x\in\fC$
\begin{align*}
\al(g)(\1\cdot\0\0\cdot x)&=B_\1(\0\0\cdot x)
=\0\1\cdot x,\\
\al(g)(\1\cdot\0\1\cdot x)&=B_\1(\0\0\cdot x)=\1\0\cdot x,~\textnormal{and}\\
\al(g)(\1\cdot\1\0\cdot x)&=B_\1(\1\0\cdot x)=\1\1\0\0\cdot x.
\end{align*}
Similar to before, this means $\al(g)$ maps $[\frac{1}{2},\frac{5}{8}]$ to $[\frac{1}{4},\frac{1}{2}]$, $[\frac{5}{8},\frac{3}{4}]$ to $[\frac{1}{2},\frac{3}{4}]$, and $[\frac{3}{4},\frac{7}{8}]$ to $[\frac{3}{4},\frac{13}{16}]$, all in the unique affine ways. Now $$\al(g)(\1\cdot\1\1\cdot x)=B_\1A_\1^2(x)=B_\1^{n+1}(x)=A_\1^2B_\1(x)=\1\1\cdot B_\1(x),$$
so by continuing to refine $x$ according to the leaves of $\tau_3$ we deduce that $\ga(g)$ looks like: 
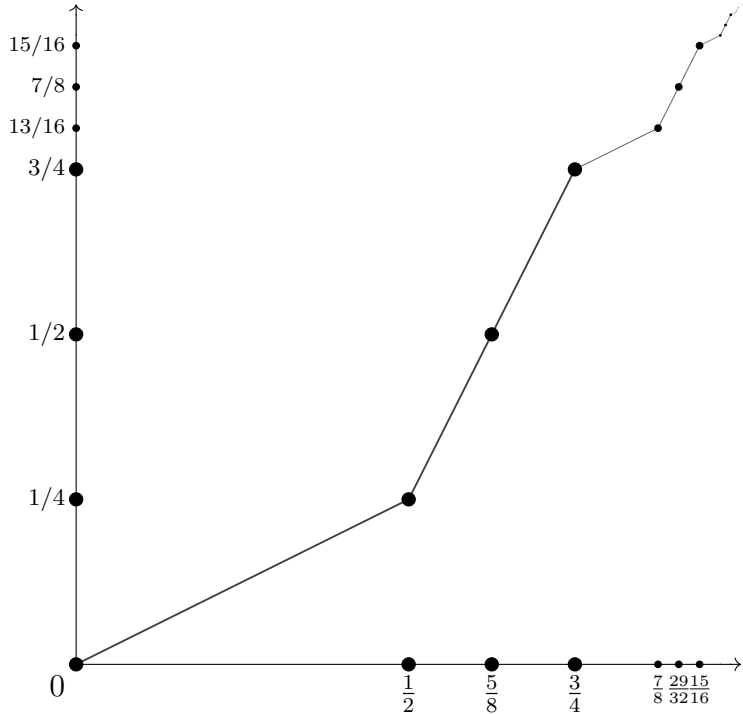
\begin{figure}[H]
\centering
\InfinitePL
\label{fig:IF}
\end{figure}

There are some interesting features in this graph. Firstly, we have infinitely many ``fake breakpoints'' (at $\frac{5}{8},\frac{29}{32},\dots$) because the slopes either side are equal. 
Most importantly, there are infinitely many genuine breakpoints (at $\frac{1}{2},\frac{3}{4},\dots$) that are accumulating at $1$. One can draw a similar graph for $[Y_b/Y_a]$ but thought of as an element of Cleary's irrational slope Thompson groups (usually denoted $F_\tau,T_\tau,V_\tau$). This also yields a graph having infinitely many genuine breakpoints. However, the action $\ga:T_\tau\act\bS$ we deduce for Cleary's group is actually conjugate to a \emph{finite} piecewise affine action (because of McCleary--Rubin rigidity). It was partially this which motivated us to investigate the question of the existence of piecewise affine actions for FS groups: is this infinite piecewise affine action an artefact of our formalism or is it something inherant to certain FS groups? In subsection \ref{sec:PP-actions} we will show that not only is $\ga:G\act\bS$ not conjugate to a finite piecewise affine action, but that $G$ actually admits no such actions at all. Finally, there is fractal behaviour happening at the point $(1,1)$. The pattern of the graph given by the mapping $[\frac{1}{2},\frac{7}{8}]\to [\frac{1}{4},\frac{13}{16}]$ is repeated in $[\frac{7}{8},\frac{15}{16}]\to[\frac{13}{16},\frac{15}{16}]$ after being scaled by $\frac{1}{4}$, and this continues ad infinitum. 

\subsection{Non-existence of piecewise projective actions}\label{sec:PP-actions} In this section we gather the results we need to obstruct our groups having piecewise projective actions. In the next section we will then prove the \ref{theo:main}.
We keep the setup of the previous section, that is, $n\geqslant3$ and we have FS groups $L_n\subset G_n$ as defined at the start of section \ref{sec:examples}, with McCleary--Rubin rigid action $\ga_n:G_n\act\bS$. We drop $n$ if it is clear from context.

\subsubsection{Free subgroups of FS groups}\label{sec:free-subgroup}

We are first going to show that a certain subgroup of $G$ contains a non-abelian free subgroup. It is perhaps not surprising that $G$ contains non-abelian free subgroups as it is a group of homeomorphisms on the circle, see for instance \cite[Prop.~4.5]{Ghys01}, however the type of embedding we produce will obstruct the existence of piecewise projective actions. Let $K\subset L$ be the set of all elements fixing a neighbourhood of $0\in\bS$. Hence, $K=K^-\cap K^+$, where $K^\pm$ is the kernel of $\ov c^\pm:L\onto\Ga^\pm$. 

\begin{proposition}\label{prop:D(K)-simple}
If the canonical action of $G$ is faithful, then $D(K)$ is simple.
\end{proposition}

This follows from the classical Higman--Epstein argument applied to a group admitting a sufficiently transitive action, for details see \cite[Thm.~3.5]{Brothier-Seelig24a}.

{\bf A non-abelian free quotient of $D(L)$.} Recall the quotient $\ov c^+:L\onto\Ga^+$ defined in subsection \ref{sec:quotients}. For the FS category in question we have $\Ga^+=\Gr\la a,b|a^2=b^n\ra$ where $n\geqslant3$. Modding out by the relation $a^2=e$ we deduce the quotient $\Ga^+\onto\Z_2*\Z_n$. Let $\pi:L\onto\Z_2*\Z_n$ be the composition of this quotient with $\ov c^+$. Now, it is standard that this induces a surjective group morphism between the derived subgroups 
\begin{align*}
\ov\pi:D(L)\onto D(\Z_2*\Z_n).  
\end{align*}
By \cite[Chap.~1, Sec.~1.3, Prop.~4]{Serre80} it follows that $D(\Z_2*\Z_n)$ is a free group of rank $n-1\geqslant2$ having basis given by the commutators $\{[a,b^i]:0<i<n\}$.

{\bf A non-abelian free subgroup of $D(L)$.} Using the universal property of free groups
\begin{align*}
\sigma:\{[a,b^i]:0<i<n\}\to D(L),~
[a,b^i]\mapsto[g_{1,a},g_{i,b}]
\end{align*}
extends uniquely to a group morphism
\begin{align*}
\ov\sigma:D(\Z_2*\Z_n)\to D(L).
\end{align*}
Since $\ov c\circ\sigma=\id$ by definition, we deduce that $\ov c\circ\ov\sigma=\id$, and so $\ov\sigma$ is injective. As such, $D(L)$ contains a free subgroup of rank $n-1\geqslant2$. Similarly, one can show $\Ga^+$ contains a non-abelian free subgroup. By proposition \ref{prop:rigid-germ} some groups of germs of the rigid action are non-amenable, which is something not shared by many other Thompson-like groups, which often have solvable groups of germs, see \cite{Brin96,Monod13}.

{\bf Embedding $L$ into $K$.} For any tree $t$ set $\iota(t):=Y_a(Y_a\ot I)(I\ot t\ot I)$. Hence, the local actions of $\iota(t)$ at its first and last leaves are $A_\0A_\0$ and $A_\1$, respectively, \emph{regardless of $t$}. Hence, if $t$ and $s$ are trees with the number of leaves, then so are $\iota(t)$ and $\iota(t)$ and  $[\iota(t)/\iota(s)]$ fix the cones $\0\0\cdot\fC$ and $\1\cdot\fC$ under the canonical action. Thus
\begin{align*}
\ov\iota:L\to K,~[t/s]\mapsto[\iota(t)/\iota(s)]
\end{align*}
is well-defined. 
This map is the restriction of a conjugation map in the fraction groupoid containing $L$ and so it automatically an injective group morphism. 

Heuristically, $\ov\iota:L\into K$ takes an element of $g\in L$ thought of as acting on the unit interval $[0,1]$, conjugates it by $\psi:x\mapsto\frac{1}{4}x+\frac{1}{4}$, and extends it into $\ov\iota(g)$ by fixing the complement of $\psi([0,1])=[\frac{1}{4},\frac{1}{2}]$. For example in $\cF_3$ the element $\ov\iota([Y_b/Y_a])$ has graph:
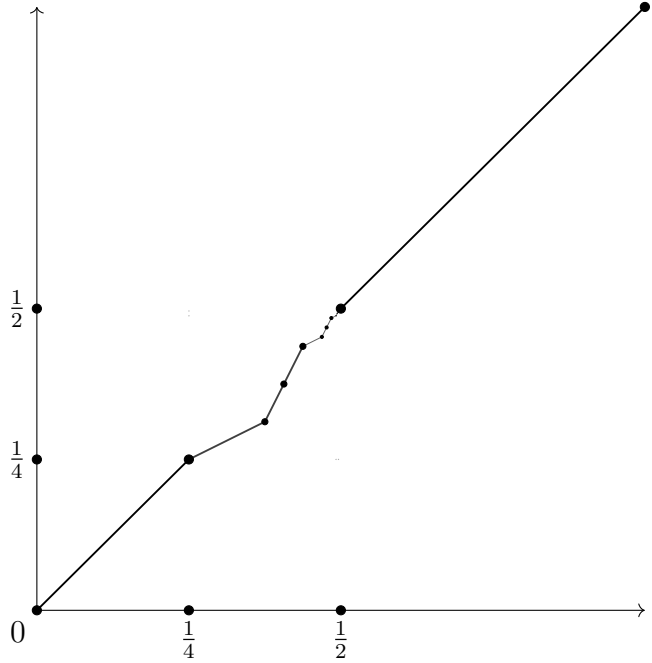
\begin{figure}[H]
\centering
\BoundedSupport
\label{fig:IF}
\end{figure}
which is a scaled and translated version of graph in subsection \ref{sec:graph}.

{\bf A non-abelian free subgroup of $D(K)$.} The embedding $\ov\iota:L\to K$ we have just constructed restricts to an embedding $\ov\iota:D(L)\to D(K)$. Hence 
\begin{align}\label{eqn:free-embedding}
\ov\iota\circ\ov\sigma:D(\Z_2*\Z_n)\to D(K)
\end{align}
is an embedding of the free group of rank $n-1\geqslant2$ into $D(K)$. 

\subsubsection{Bounded cohomology and actions on the circle}

We now seek to understand particular features of \emph{all} actions by orientation-preserving homeomorphisms (OPH actions) $D(K)\act\bS$. The tool suited for our needs is \emph{bounded cohomology}, for a nice overview see \cite{Ghys01} and for more details see for instance \cite{Frigerio17}. A classic result of Ghys states that the second integral bounded cohomology group $H^2_b(\Gamma;\Z)$ classifies OPH actions $\Gamma\act\bS$ up to \emph{semi-conjugacy}, which is a certain weakening of conjugacy, see \cite[Prop.~5.4]{Ghys87} or \cite[Thm.~6.6]{Ghys01}. In particular:

\begin{theorem}[Ghys `87]\label{theo:Ghys} If $H^2_b(\Ga;\Z)=0$, then any OPH action $\eta:\Ga\act\bS$ is semi-conjugate to the trivial action, hence $\eta:\Ga\act\bS$ has a global fixed point.
\end{theorem}

Recently, Fournier-Facio and Lodha have developed a powerful algebraic criterion for a group $\Ga$ to have trivial second \emph{real} bounded cohomology group $H^2_b(\Ga;\R)=0$, see \cite[Thm.~1.2]{Fournier-Facio-Lodha23}. The condition on $\Ga$ is called having \emph{commuting conjugates}: for any finitely generated subgroup $\Delta\subset\Ga$ there exists $\ga\in\Ga$ such that every element of $\Delta$ commutes with every element of the conjugate $\ga\Delta\ga^{-1}$. 
When $\Ga$ is also perfect (i.e.,~has no proper abelian quotients), the authors deduce a stronger result concerning integral (rather than real) bounded cohomology:

\begin{theorem}[Fournier-Facio--Lodha `23]\label{theo:Fournier-Facio--Lodha}
If $\Ga$ is a perfect group with commuting conjugates, then $H_b^2(\Ga;\Z)=0$.
\end{theorem}

See the proof of \cite[Cor.~4.20]{Fournier-Facio-Lodha23} for details. 

\begin{proposition}\label{prop:global-fixed-point}
The group $D(K)$ is perfect and has commuting conjugates. Hence $H^2_b(D(K);\Z)=0$ and any OPH action $D(K)\act\bS$ has a global fixed point.
\end{proposition}
\begin{proof}
By theorem \ref{theo:faithful} the canonical action $\al:G\act\fC$ is faithful, so proposition \ref{prop:D(K)-simple} implies $D(K)$ is simple, and thus perfect. Let $\Delta$ be the subgroup of $D(K)$ generated by some non-trivial elements $\ga_1,\dots,\ga_l\in D(K)$. By definition, each $\ga_i$ fixes a proper clopen  neighbourhood $C_i$ of $\{\ov\0,\ov\1\}$. 
Hence, $C=\cap_{i=1}^lC_i$ is a proper clopen neighbourhood of $\{\ov\0,\ov\1\}$ fixed by every element of $\Delta$. Let $D:=\fC\setminus C$, which is a non-empty clopen set, and let $E$ be a non-empty clopen subset of $C$ that doesn't contain $\ov\0$ and $\ov\1$. By claim 2 of \cite[Thm.~3.5]{Brothier-Seelig24a} there exists $\ga\in D(K)$ sending $D$ into $E$. It follows that any element of $\ga\Delta\ga^{-1}$ commutes with any element of $\Delta$. By theorem \ref{theo:Fournier-Facio--Lodha} we have $H_b^2(D(K);\Z)=0$, so by theorem \ref{theo:Ghys} any OPH action $D(K)\act\bS$ has a global fixed point.
\end{proof}

\subsubsection{Piecewise defined groups and their subgroup structure}\label{sec:piecewise}

Consider the real plane $\R^2$ equipped with its euclidean topology. The quotient of the subspace $\R^2\setminus\{(0,0)\}$ by the equivalence $(x,y)\sim(\lambda x,\lambda y)$, $\lambda\not=0$, is the \emph{real projective plane}, which we denote by $\bP$ equip with the quotient topology. We write the equivalence class of $(x,y)$ by $[x:y]$ and identify the real line $\R$ inside $\bP$ as the set of all points of the form $[x:1]$ and we denote the remaining point $\infty:=[1:0]$. The projective line is homeomorphic to the unit circle $\bS\subset\R^2$ via stereographic projection. As such, $\bP$ inherits a cyclic order (equivalently, an orientation) from $\bS$ and thus the notion of intervals. The group $\SL_2(\R)$ of $2\times2$ real matrices with determinant $1$ acts OPH on $\bP$ via the formula
\begin{align*}
\begin{bmatrix}
a&b\\
c&d
\end{bmatrix}
\cdot[x:y]=[ax+by:cx+dy].
\end{align*}
The kernel of this action is given by $\{\pm I\}$, where $I$ is the identity matrix. Hence, the quotient $\PSL_2(\R):=\SL_2(\R)/\{\pm I\}$ acts faithful OPH on $\bP$. The stabiliser of $\infty$ consists of all elements represented by matrices
\begin{align*}
\begin{bmatrix}
\sqrt{a}&\frac{b}{\sqrt{a}}\\
0&\frac{1}{\sqrt{a}}\\
\end{bmatrix},
\end{align*}
where $a>0$ and $b\in\R$. On $\R\subset\bP$ this matrix acts as the affine map $x\mapsto ax+b$. Hence, the stabiliser of $\infty$ in $\PSL_2(\R)$ is called the (orientation-preserving) \emph{affine group} and denote it by $\Aff_+(\R)$ (the symbol $+$ refers to ``orientation-preserving'').

{\bf Piecewise projective homeomorphisms.} We say a homeomorphism $\ga:\bP\to\bP$ is \emph{piecewise projective} if there are \emph{finitely many} intervals $I_1,\dots,I_n$ covering $\bP$ and elements $\ga_1,\dots,\ga_n\in\PSL_2(\R)$ such that $\ga|_{I_i}=\ga_i|_{I_i}$ for all $i=1,\dots,n$. 
The set $\PP_+(\bP)$ of all piecewise projective orientation-preserving homeomorphisms of $\bP$ is closed under composition, inverses, and contains the identity, so forms a group. 
The stabiliser subgroup of $\infty$ inside $\PP_+(\bP)$ is denoted $\PP_+(\R)$ which is the group of \emph{piecewise projective orientation-preserving homeomorphisms of $\R$}. 

{\bf Piecewise affine homeomorphisms.} Similarly, a homeomorphism $\ga:\bP\to\bP$ is \emph{piecewise affine} if it is piecewise projective but the elements $\ga_1,\dots,\ga_n$ defining the action of $\ga$ can be chosen from $\Aff_+(\R)$. The set of all piecewise affine homeomorphisms of $\bP$ is denoted $\PL_+(\bP)$ and it is also a group under composition. It is denoted $\PL$ instead of $\textnormal{PA}$ for historical reasons --- it stands for ``piecewise linear''. The stabiliser of $\infty$ inside $\PL_+(\bP)$ is denoted $\PL_+(\R)$ which exactly coincides with the usual group of piecewise affine homeomorphisms as defined in \cite{Brin-Squier85}. By definition we have $\PL_+(\R)\subset\PP_+(\R)$ and $\PL_+(\bP)\subset\PP_+(\bP)$. Our obstruction to acting piecewise affine and piecewise projective uses a famous theorem of Brin and Squier pertaining to the subgroup structure of $\PL_+(\R)$, nicely adapted to $\PP_+(\R)$ by Monod, see \cite{Brin-Squier85,Monod13}.

{\bf Connection between groups acting on the projective line and on the circle.} Consider $\bS=\R/\Z$ and write $\phi:[0,1]\to \bS$ for the usual quotient map.
Here, a homeomorphism $\ga:\bS\to\bS$ is called \emph{piecewise affine} if there exist finitely many intervals $I_1,\dots,I_n$ covering $[0,1]$ and satisfying that $\phi^{-1}\gamma\phi$ restricted to any of the $I_j$ is an affine map (hence of the form $x\mapsto ax+b$ with $a,b\in\R$).
We write $\PL_+(\bS)$ for the group of orientation-preserving piecewise affine homeomorphisms of $\bS$. As mentioned, the stereographic projection furnishes a homeomorphism $\varphi:\bP\to\bS$. Since this takes intervals to intervals and affine maps to affine maps, conjugation induces a group isomorphism $\PL_+(\bP)\to\PL_+(\bS)$, hence recovering the more standard group of piecewise affine homeomorphisms of the circle. 
One may then define the so-called group of piecewise projective orientation-preserving homeomorphisms of the circle $\PP_+(\bS)$ as being the conjugated group $\varphi \PP_+(\bP)\varphi^{-1}$ (noting that conjugating by $\varphi$ sends $\PL_+(\bP)$ to $\PL_+(\bS)$), see \cite{Cornulier21} for further details.

\begin{theorem}[Brin--Squier `85, Monod `13]\label{theo:Brin-Squier} Neither $\PL_+(\R)$ nor $\PP_+(\R)$ contain a non-abelian free subgroup.
\end{theorem}

\subsection{Proof of the main result}\label{sec:main-theorem} We are now able to prove the \ref{theo:main}. Fix $n\geqslant3$ and let $\cF_n,G_n,\al_n:G_n\act\fC,\ga_n:G_n\act\bS,K_n$ be as defined throughout the above section. As usual we drop $n$ subscripts if it is clear from context.

\subsubsection{Simple groups of type $\mathrm{F}_\infty$ acting by orientation-preserving homeomorphisms on the circle} By theorem \ref{theo:faithful}, the canonical action $\al:G\act\fC$ is faithful, so by theorem \ref{theo:simple} we deduce that $D(G)$ is simple. Hence, the rigid action $\ga:G\act\bS$ is faithful OPH (i.e.,~act by orientation-preserving homeomorphisms). Since $\cF=\FS\la a,b|\tau(a)=\rho(b)\ra$, where $\tau(a)$ and $\rho(b)$ both have $n$ carets, theorem \ref{theo:abelianisation} implies the abelianisation of $G$ is $\Z_n$. Since $\Z_n$ is finite and $G$ is of type $\mathrm{F}_\infty$ by theorem \ref{theo:finiteness-properties}, it follows that $D(G)$ is of type $\mathrm{F}_\infty$ also, see \cite[Cor.~7.2.4]{Geoghegan07}. In particular, $D(G)$ is finitely presented.

\subsubsection{No piecewise projective actions} For a contradiction, suppose $\eta:D(G)\act\bP$ is a non-trivial piecewise projective action. Since $D(G)$ is simple, it follows that $\eta$ is faithful. It follows that $\eta$ restricts into a faithful piecewise projective action $D(K)\act\bP$. This is (conjugate to) a faithful OPH action on the circle, so by proposition \ref{prop:global-fixed-point}, the action has a global fixed point. Up to conjugating the action by an element of $\PSL_2(\R)$ (which keeps the action piecewise projective) we may assume that $\infty$ is a global fixed point for $D(K)\act\bP$. As the action is faithful, and $\PP_+(\R)\subset\PP_+(\bP)$ is the stabiliser of $\infty$, we deduce a faithful action $D(K)\act\R$ by piecewise projective homeomorphisms. However, the embedding \ref{eqn:free-embedding} of a non-abelian free group inside $D(K)$ implies a non-abelian free group embeds inside $\PP_+(\R)$ via the faithful action $D(K)\act\R$. This contradicts theorem \ref{theo:Brin-Squier}, and so there cannot be any non-trivial piecewise projective actions $\eta:D(G)\act\bP$.

\subsubsection{Infinitely many isomorphism classes} 

Finally we are going to show $D(G_n)$ and $D(G_m)$ are isomorphic only if $n=m$. 
We describe all the groups of germs for the actions $\ga_n:D(G_n)\act \bS$. Recall that $\Q_2$ denotes the dyadic rationals inside $\bS=\R/\Z$ and let $\Q$ be the rational points inside $\bS$.
Firstly, from propositions \ref{prop:canonical-germ} and \ref{prop:rigid-germ} we deduce
\begin{align}\label{eqn:germ-list}
[D(G_k)]_x\simeq
\begin{cases}
\Z\times\Ga_k^+&\textnormal{if $x\in\Q_2$,}\\
\Z&\textnormal{if $x\in\Q\setminus\Q_2$,}\\
\{e\}&\textnormal{if $x\in\bS\setminus\Q$,}
\end{cases}
\end{align}
where $\Ga^+_k=\Gr\la a,b|a^2=b^k\ra$ for all $k\geqslant3$. Observe the abelianisation of $\Z\times\Ga^+_k$ is $\Z\times\Z$, hence $\Z\times\Ga^+_k$ cannot be isomorphic to $\Z$ nor to $\{e\}$. Fix $n,m\geqslant3$, and suppose $\phi:D(G_n)\to D(G_m)$ is an isomorphism. By theorem \ref{theo:rigid} the actions $\ga_n:D(G_n)\act\bS$ and $\ga_m:D(G_m)\act\bS$ we constructed in section \ref{sec:dynamics} are McCleary--Rubin rigid. Hence, by theorem \ref{theo:C-rigid} there exists a unique homeomorphism $\Phi:\bS\to\bS$ satisfying $\ga_m(\phi(g))=\Phi\circ\ga_n(g)\circ\Phi^{-1}$ for all $g\in D(G_n)$. As a result, we deduce an isomorphism
\begin{align*}
\phi_x:[D(G_n)]_x\to[D(G_m)]_{\Phi(x)},~[g]_x\mapsto[\Phi\circ g\circ\Phi^{-1}]_{\Phi(x)}
\end{align*}
between groups of germs for every $x\in\bS$. Suppose $x\in\Q_2$. The list of groups of germs \ref{eqn:germ-list} and the above observation forces $\Phi(x)\in\Q_2$. Hence, we deduce an isomorphism
\begin{align*}
\phi_x:\Z\times\Ga_n^+\to\Z\times\Ga_m^+.
\end{align*}
Now, the centre of a direct product is the direct product of centres. Hence, $Z(\Z\times\Ga^+_n)=Z(\Z)\times Z(\Ga^+_n)=\Z\times Z(\Ga^+_n)$. Moreover, the centre of $\Ga^+_n$ is the subgroup generated by $a^2=b^n$, so we deduce $\Z\times\Ga^+_n/Z(\Z\times\Ga^+_n)\simeq\Ga_n^+/Z(\Ga_n^+)\simeq\Z_2*\Z_n$. As the centre of a group is a characteristic subgroup, $\phi_x$ passes to an isomorphism between quotients
\begin{align*}
\ov\phi_x:\Z_2*\Z_n\to\Z_2*\Z_m.
\end{align*}
Finally, taking abelianisations, we deduce $\Z_2\times\Z_n\simeq\Z_2\times\Z_m$, so $n=m$. This completes the proof of the \ref{theo:main}.

\subsection{Comparison with other groups}
We finish the article by giving a comparison between our groups and two other classes of groups.

\subsubsection{Lodha's group}

As mentioned in the introduction, Lodha defined a variant $\Lambda$ of the Lodha--Moore groups that acts naturally on $\bP$ by piecewise projective homeomorphisms, and showed it was finitely presented, simple, and admits no non-trivial piecewise affine actions \cite{Lodha-Moore16,Lodha19}. One quick definition of $\Lambda$ is that it is generated by a natural copy of Thompson's group $T$ inside $\PP_+(\bP)$ together with $c\in\PP_+(\R)$ defined by
\begin{align*}
c(x):=
\begin{cases}
x&\textnormal{if $x\leqslant0$,}\\
\frac{2x}{1+x}&\textnormal{if $0\leqslant x\leqslant1$,}\\
\frac{2}{3-x}&\textnormal{if $1\leqslant x\leqslant2$,}\\
x&\textnormal{if $x\geqslant2$,}
\end{cases}
\end{align*}
see \cite{Lodha-Zaremsky23} where $\Lambda$ is denoted $S$. 
Since $T$ contains non-abelian free subgroups, so does $\Lambda$.
It was later shown by Lodha and Zaremsky that $\Lambda$ is of type $\mathrm{F}_\infty$, see \cite[Thm.~6.3]{Lodha-Zaremsky23}. Initially, we were only interested in non-existence of piecewise affine actions for FS groups, but after becoming aware of Lodha's work we investigated piecewise projective actions and it turned out our obstruction was strong enough to obstruct piecewise projective actions also. Despite $\Lambda$ admitting a piecewise projective action, the proof that $\Lambda$ does not admit a piecewise affine action shares similarities with ours. 
Lodha's piecewise affine obstruction is due to a specific embedding $BS(1,2)\into\Lambda$, where $BS(m,n):=\Gr\la a,b|ba^mb^{-1}=a^n\ra$ is the $(m,n)$-Baumslag--Solitar group, see \cite[Prop.~5.1]{Lodha19}. 
The argument goes: if $\eta:\Lambda\act\bS$ were a piecewise affine action, then by a theorem of Guelman--Liousse \cite[Prop.~2.3]{Guelman-Liousse11}, there exists an embedding $\iota:BS(1,2^m)\into\Lambda$ (for a certain $m\geqslant 1$) for which $\iota(BS(1,2^m))\act \bS$ has a global fixed point under $\eta$. Hence, $\eta$ induces a piecewise affine action of $BS(1,2^m)$ on the unit interval $[0,1]$, however no such action exists due to a variant of the Brin--Squier theorem due to Guba--Sapir \cite[Cor.~23]{Guba-Sapir99}. 
Using similar arguments, Lodha shows $\Lambda$ does not admit a $C^1$ (continuously differentiable) action on $\bS$. This is the first finitely presented simple group witnessing this behaviour. We currently do not know any FS groups that cannot admit a $C^1$ action on $\bS$. 
However, note all $n$-ary Higman--Thompson groups $F_n$ can be realised as FS groups, and these groups admit a $C^\infty$ action on $\bS$ due to Ghys--Sergiescu \cite{Ghys-Sergiescu87}. 
On the geometric side of things, the embedding $BS(1,2)\into\Lambda$ implies $\Lambda$ cannot act properly on a CAT(0) cube complex by \cite[Thm.~1.2]{Haglund23}. Though, it is still unknown if $\Lambda$ has the Haagerup property (which is equivalent of having Gromov's a-T-menability and is a consequence of acting properly on a CAT(0) cube complex). At this moment we still don't have much insight regarding FS groups and properties such as acting $C^1$ on the circle, acting properly on a CAT(0) cube complex, or having the Haagerup property.

\subsubsection{Finite germ extensions}\label{sec:FGE}

As discussed in subsection \ref{sec:graph}, the rigid action $\ga:G_3\act \bS$ produces a homeomorphism $\ga(g)$ of $\bS$ that acts piecewise affine except around the point $1$ and having infinitely many breakpoints which accumulate at $1.$ In fact, taking any element $g\in G_3$ we obtain a homeomorphism $\alpha(g)$ with possibly infinitely many breakpoints accumulating in a \emph{finite} subset $X(g)\subset \bS$ so that $\alpha(g)$ is locally affine on the complement of $X(g)$. This behaviour is strikingly similar to the axioms defining \emph{finite germ extensions}, recently introduced by Belk, Hyde, and Matucci, \cite{Belk-Hyde-Matucci24}. Given a Hausdorff space $X$ and groups $B\subset\Ga\subset\Homeo(X)$, one calls $\Ga$ a finite germ extension (FGE) over $B$ if
\begin{itemize}
\item every element of $\Ga$ acts locally like some element of $B$ everywhere in $X$ except at \emph{finitely many} ``singular points'';
\item $B$ can be characterised as the subgroup of $\Ga$ whose elements have no singular points; and
\item for any $\ga\in\Ga$ and a singular point $x$ of $\ga$ there exists $\ga'\in\Ga$ whose \emph{unique} singular point is $x$ and that acts locally the same as $\ga$ at $x$. 
\end{itemize}	
In \cite{Belk-Hyde-Matucci24}, the authors deduce results for FGE's about their finiteness properties, simplicity, abelianisation, and also realise many existing groups as FGE's, for instance, Cleary's irrational slope Thompson groups and R\"{o}ver--Nekrashevych groups coming from bounded automata groups. With this new perspective the authors resolve a conjecture of Nekrashevych regarding the finiteness properties of the aforementioned class of R\"{o}ver--Nekrashevych groups. Moreover, the authors produce two new type $\mathrm{F}_\infty$ (in particular, finitely presented) simple groups $T\cA$ and $V\cA$ as FGE's which contain every countable abelian group, hence resolving the Boone--Higman conjecture for countable abelian groups. As hinted at above, each member of our family of examples $(G_n)_{n\geqslant3}$ can be realised as FGE's acting on the circle over their respective subgroups of $a$-coloured diagrams (isomorphic to Thompson's group $T$). Hence, $B=T$, $\Ga$ is our FS group $G_n$, $X=\bS$, and the singular points of $g\in G_n$ is the set $X(g)$ of accumulation points of the breakpoints. However, we do not know if every FS group (with faithful canonical action) can be realised in this way. Detailed proofs of these claims above will appear in a future article. We are hopeful that studying certain FS groups as FGE's may reveal new properties of FS groups and that these connections may be beneficial to both theories.

\newcommand{\etalchar}[1]{$^{#1}$}

\end{document}